\documentclass[reqno]{amsart}
\usepackage{amsmath, amsthm, amssymb, amstext}
\usepackage[foot]{amsaddr}
\usepackage[left=2cm,right=2cm,top=1.5cm,bottom=1.5cm,includeheadfoot]{geometry}
\usepackage{hyperref,xcolor}
\hypersetup{
	pdfborder={0 0 0},
	colorlinks,
}

\usepackage{todonotes}
\usepackage{enumitem}
\newtheorem{theorem}{Theorem}
\newtheorem{remark}[theorem]{Remark}
\newtheorem{lemma}[theorem]{Lemma}
\newtheorem{proposition}[theorem]{Proposition}

\newtheorem{definition}[theorem]{Definition}

\DeclareMathOperator*{\meas}{meas}          %

\newcommand{\DD} {\displaystyle}
\newcommand{\la} {\lambda}

\newcommand{\Om}{\Omega}

\newcommand{\ds} {\displaystyle}
\newcommand{\e}{\epsilon}

\newcommand{\al} {\alpha}
\newcommand{\ba} {\beta}

\newcommand{\ga} {\gamma}

\newcommand{\ra} {\rightarrow}

\newcommand{\De} {\Delta}

\newcommand{\noi} {\noindent}
\newcommand{\na} {\nabla}

\newcommand{\mb} {\mathbb}
\newcommand{\mc} {\mathcal}

\numberwithin{theorem}{section} \numberwithin{equation}{section}

\title[$p$-Kirchhoff problem with Stein-Weiss nonlinearity ]{ Multiplicity results for  $p$-Kirchhoff modified Schr\"odinger equations with Stein-Weiss type critical nonlinearity in $\mb R^N$ }

\author[Reshmi Biswas, Sarika Goyal and  K. Sreenadh]
{Reshmi Biswas, Sarika Goyal and K. Sreenadh}
\address{Reshmi Biswas \newline
	Department of Mathematics, IIT Delhi, Hauz Khas, New Delhi 110016, India}
\email{reshmi15.biswas@gmail.com}

\address{Sarika Goyal \newline
	Department of Mathematics, Bennett University, Greater Noida, Uttar Pradesh 201310, India}
\email{sarika.goyal@bennett.edu.in }

\address{K. Sreenadh \newline
	Department of Mathematics, IIT Delhi, Hauz Khas, New Delhi 110016, India}
\email{sreenadh@maths.iitd.ac.in}


\begin{document}
	\begin{abstract}
		In this article, we consider the following modified quasilinear  critical Kirchhoff-Schr\"odinger problem involving Stein-Weiss type nonlinearity:
		$$  \quad \left\{
		\begin{array}{lr}
			\mc K(u)=  \la f(x) |u(x)|^{q-2} u(x)+\left(\ds\int_{\mathbb R^N}\frac{|u(y)|^{2p_{\beta,\mu}^{*}}}{|x-y|^{\mu}|y|^{\beta}}dy\right)\frac {|u(x)|^{2p_{\beta,\mu}^{*}-2} u(x)}{|x|^\beta}  \; \text{in}\; \mathbb R^N
		\end{array}
		\right.
		$$
		where \(\la>0\)  is a parameter, \(N\geq 3\), $\mc K(u) = \left(a+b\ds\int_{\mathbb R^N}|\nabla u|^{p}dx\right)\De_{p} u - a u \Delta_{p}(u^2)$ with \(a>0\), \(b\geq 0\),  $\beta\geq0,$ $0<\mu<N$, \(0<2\ba+ \mu< N\), \(2\leq q< 2 p^*\). Here $p^*=\frac{Np}{N-p}$ is the Sobolev critical exponent and \( p_{\ba,\mu}^{*}:= \frac p2\frac{(2N-2\beta-\mu)}{N-2} \) is the critical exponent with respect to the doubly weighted Hardy-Littlewood-Sobolev inequality (also called Stein-Weiss type inequality).  Then by establishing  a concentration-compactness argument for our problem, we show the existence of infinitely many nontrivial solutions to the equations with respect to the parameter \(\la\)   by using Krasnoselskii’s genus theory, symmetric mountain pass theorem and $\mb Z_2$- symmetric version of mountain pass theorem for different range of $q$. We further show  that these solutions belong to $L^\infty(\mb R^N)$.
	\end{abstract}
	
	\maketitle
	\section{\bf Introduction}
	Our aim in this article is to  study the following modified quasilinear critical Kirchhoff-Schr\"odinger problem involving Stein-Weiss type critical nonlinearity:
	\begin{equation}\label{eq1}  \quad \left\{
		\begin{array}{lr}
			\mc K(u)= \la f(x) |u(x)|^{q-2} u(x)+\left(\ds\int_{\mathbb R^N}\frac{|u(y)|^{2p_{\beta,\mu}^{*}}}{|x-y|^{\mu}|y|^{\beta}}dy\right)\frac {|u(x)|^{2p_{\beta,\mu}^{*}-2} u(x)}{|x|^\beta}  \; \text{in}\; \mathbb R^N
		\end{array}
		\right.
	\end{equation}
	where  
	$\mc K(u) = \left(a+b\ds\int_{\mathbb R^N}|\nabla u|^{p}dx\right)\De_{p} u - a u \Delta_{p}(u^2)$, 
	 $2\leq p<N$, \(a>0\), \(b\geq 0\), $\beta\geq 0,$ $\mu>0,$ \(0<2\beta+\mu<N \), {\(0< p_{\ba,\mu}^{*}:= \frac{p(2N-2\ba-\mu)}{2(N-p)} \), \(N\geq 3\)} and $\la>0$ is a parameter.   Here $2p<q<2p^*$, $p^*:=\frac {Np}{N-p}$ and $f(\geq 0)\in L^{\frac{2p^*}{2p^*-q}}(\mathbb R^N)$.\\

	\noi  The solutions of \eqref{eq1} involving the Schr\"odinger operator   $-\Delta_p u -u\Delta_p (u^{2}),$   are related with the solitary standing wave solutions to the  quasilinear Schrödinger equation 
of	the form
	\begin{align}\label{l}
		iu_t =-\Delta u+V(x)u-h_1(|u|^2)u-C\Delta h_2(|u|^2)h_2'(|u|^2)u,\;\; x\in\mathbb R^N,
	\end{align}
	where $V:\mb R^N\to \mb R$ is a continuous potential function, $C>0$ is  some positive real constant, $h_1$ and $h_2$ are some real valued functions with some appropriate assumptions. Based upon the different forms of the function   $h_2$, \eqref{l} explains  different phenomenon in the mathematical physics. For example, if $h_2(s)=s$ (see \cite{1}), then \eqref{l} is used in modelling the superfluid film equation in plasma physics and if $h_2=\sqrt{1+s^2}$ (see \cite{33}), \eqref{l} represents the self-channeling of a high-power ultra short laser in matter. Such kind of equations also have applications in the modeling of   dissipative quantum mechanics \cite{has}, plasma physics and fluid mechanics \cite{bass}, etc.\\\\
	The main feature of such operator is that the term $u\Delta_p(u^2)$, present in  \eqref{eq1},  does not let the natural energy functional corresponding to  \eqref{eq1} to be well defined for all $u\in D^{1,p}({\mathbb R^N})$ (defined in Section \ref{sec2}). Therefore, the  standard critical point theory in variational method is inconvenient to apply directly for such problems of type \eqref{eq1}. To overcome this 
	inconvenience, researchers have established several techniques and arguments, such as    constrained minimization
	technique (see \cite{29}), the perturbation method (see \cite{24}),
	change of variables (see \cite {CJ,do}). In this article, we use the change of variable method described in Section \ref{sec2}. In the recent advancement on such modified quasilinear equations,  we refer  the readers to  \cite{my6,8,12,bz3,lgg,my5} and the cited research works there in with no claim on completeness.\\\\
	%
	The nonlocal nonlinearity present in \eqref{eq1} is inspired  by the doubly weighted Hardy-Littlewood-Sobolev inequality (also called Stein-Weiss type inequality). Let us first recall the  well-known doubly weighted Hardy-Littlewood-Sobolev inequality (see \cite{SW}), which is stated as:
		\begin{proposition}\label{P1}
		 Let $t$, $s>1$, $\mu>0 $, $\vartheta+\beta\geq0$ and $0<\vartheta+\beta+\mu< N.$ Also, let $\frac 1t+\frac {\mu+\vartheta+\beta}{N}+\frac 1s=2$, $\vartheta<\frac{ N}{t'}, \beta<\frac {N}{s'}$, $g_1 \in L^t(\mathbb R^N)$ and $g_2 \in L^s(\mathbb R^N)$, where $t'$ and $s'$ denote the H\"older conjugate of $t$ and $s$, respectively. Then there exists a  constant $C(N,\mu,\vartheta,\beta,t,s)$, independent of $g_1,$ $g_2$ such that
		\begin{equation}\label{HLSineq}
			\int_{\mb R^N}\int_{\mb R^N} \frac{g_1(x)g_2(y)}{|x-y|^{\mu}|y|^\vartheta|x|^\beta}\mathrm{d}x\mathrm{d}y \leq C(N,\mu,\vartheta,\beta,t,r)\|g_1\|_{L^t(\mb R^N)}\|g_2\|_{L^s(\mb R^N)}.
		\end{equation}
	\end{proposition}
	
	
	\noi If $\vartheta=\beta=0$, this inequality \eqref{HLSineq} is the classical Hardy-Littlewood-Sobolev inequality (see \cite{lieb}). 
	S. Pekar first stated the study of such equations in  \cite{pekar}, where the author considered 
	the  nonlinear Schr\"{o}dinger-Newton equation of the form:
	\begin{align}\label{sn}
		-\Delta u + V(x)u = ({\mathcal{K}}_\mu * u^2)u +\la f_1(x, u)\;\;\; \text{in}\; \mathbb R^N,
	\end{align}
	where $\la>0$ is a parameter, $\mathcal{K}_\mu$ is  the  Riesz potential, $V:\mb R^N\to \mb R$ is continuous potential function, $f_1:\mb R^N\times\mb R\to \mb R$ is a Carath\'eodory function with some suitable assumptions and $*$ denotes the convolution.
	These types of equations are very much crucial in the application point of view  in Physics to describe the Bose-Einstein
	condensation (see \cite{bose}), the self gravitational
	collapse of a quantum mechanical wave function (see \cite{penrose}), etc.  With the help of  the elliptic equations of type, \eqref{sn} P. Choquard (see \cite{choq}) managed to explain the quantum theory of a polaron at rest and for modeling the phenomenon when an electron gets trapped in its own hole. 
	For deeper  study  of Choquard  equations, we refer to the readers the research works
	\cite{my3,my4,lieb,Lions,moroz-main} and the references therein. When $\beta>0$, then the elliptic problems involving Stein-Weiss type nonlinearities are studied in \cite{rs2,my1,my2}, very recently. \\\\
	On the other hand,  one of the  main features of \eqref{eq1} is  the presence of the both nonlocal Kirchhoff term  in the left-hand side and nonlocal Stein-Weiss type nonlinearity in the right-hand side of \eqref{eq1}. Hence our problem
	 is categorized as a doubly nonlocal problem. The   Kirchhoff-type models  arise  in  various   physical  and  biological  systems and hence, the study of the problems involving Kirchhoff operators has been quite popular in recent years. Precisely, 
	Kirchhoff established a model given by the following equation:
	$$\rho\frac{\partial^2 u}{\partial^2t}-\left(\frac{p_0}{h}+\frac{E}{2L}\int_0^L\left|\frac{\partial u}{\partial t}\right|^2dx\right)\frac{\partial^2 u}{\partial^2t}=0,$$ 
	which extends the classical D’Alembert wave equation by taking into account the effects of the changes in the length of
	the strings during the vibrations,
	where the constants $\rho, p_0, h, E, L$ 
	represent physical parameters of the string.
	Subsequently,  using the method of  Nehari manifold and the concentration
	compactness principle   L\"{u} \cite{lu}  studied the following Kirchhoff-Choquard problem
	\begin{align}\label{sn2}
			\Big(-a+b \int_{\mb R^3}
			|\nabla u|^2 dx\Big)\Delta u + V_\la(x)u = ({\mathcal{K}}_\mu * |u|^q)|u|^{q-2}u ~~\text{in}~ \mb R^3,\end{align} 
		where $a >0,  b\ge 0,$  $\mathcal{K}_\mu$ is  the  Riesz potential,
	$ V_\la(x) = 1 + \la g(x)$, $\la>0$ and $g$ is a continuous
	potential function, $q \in (2, 6 -\mu).$ 
Later Liang et. al \cite{lg1} studied \eqref{sn2} 	for $V_\la=0$, $q=2_\mu^*$ and by adding some perturbation  in the right-hand side which has sub-critical growth 
	in the sense of Sobolev inequality. Though the literature on  the Kirchhoff equation is really vast, without attempting to provide the complete list we refer to \cite{bz1,bz2,bz4,fig2} and references there in  to the readers.\\\\
	When it comes to the Kirchhoff problems involving the operator present in our problem \eqref{eq1}, without the convolution term in the nonlinearity in \eqref{eq1}, Liang et. al \cite{lg2} studied the multiplicity results for such modified quasilinear Kirchhoff equations. Then for $p=2$ and $\beta=0$ in \cite{na}, the authors studied such problem. 
	Also, for $\beta=0$, we refer to the work in \cite{bz3}, which deals with the  Choquard equations involving the modified quasilinear Schr\"odinger operator as in \eqref{eq1}, without the Kirchhoff term.\\\\
Motivated by all the aforementioned  works, in  this  article, we consider \eqref{eq1} for  $2\leq p<\infty$ and with critical Stein-Weiss type convolution term in combination with  sub-critical perturbation. The suitable Stein-Weiss type critical exponent is set here as $2.p_{\beta,\mu}^*$ due to the  Schr\"odinger term $u\Delta_{p}(u^2)$ present in the principal operator and for the same reason the  exponent $q$ also varies between $2$ to $2p^*$. We exhibit the existence of infinitely many solutions for \eqref{eq1}  by exploiting Krasnoselskii’s genus theory and by applying a variant  of Clark's theorem (also known as $\mb Z_2$- symmetric version of mountain pass theorem). 
   One of the main contributions to this article is that we  have proved a   concentration-compactness result (see Lemma \ref{c1}) related to our problem for general $2\leq p<\infty$, which is not yet studied  even for the equations of type \eqref{eq1} without the  Schr\"odinger term $u\Delta_{p}(u^2)$. In case of $p=2$, Du et al. \cite{my1} studied such result   for the equations involving Laplacian. This concentration-compactness result plays a very crucial in the context of our problem where we face lack of compactness due to the presence of the critical exponent as well as, the domain being whole of $\mb R^N$. This result will help us to analyze the behavior of the weakly convergent sequences in the solution space $D^{1,p}(\mb R^N)$ (see Section \ref{sec2}) so that we can prove the Palaise-Smale condition for the energy functional $\mc I_\la$(see \eqref{energy}) below some critical level. Furthermore, we prove global $L^\infty$ regularity estimate on the solutions to \eqref{eq1}, which is applicable even for the critical Choquard equation $(\beta=0)$ involving $p$-Laplacian in the whole of $\mb R^N$ for $p\geq2$, which is another important contribution to this article. We would like to mention that the detailed regularity results for the critical Choquard equations involving fractional and local $p$-Laplacian is first studied in \cite{rs1} for the bounded domain in $\mb R^N$.  The main difficulty we face here is due to involvement of the Stein-Weiss type critical nonlinearity and Kirchhoff term together. This  gives rise to several interactions with the exponent $q,$ $p^*$, $p_{\beta,\mu}^*$ and the parameters $\la,a,b,\mu,\beta$ which have effects in the multiplicity results. Depending on this, we need to  carry out delicate analysis. To the best of our knowledge the results studied in this present paper are not available  in the literature for the equation of type \eqref{eq1}.\\\\
	Now we state the main results in this article. Throughout this article we take the following assumptions on the parameters:
	\begin{equation} \label{assumption1}  
		 a>0,\; b\geq 0,\; \beta\geq 0,\;\mu>0,\; 0<2\beta+\mu<\min \{2p,N\},\; 2\leq p<N. 
	 \end{equation}
	\begin{theorem}\label{main.result.1}
		Let  $2<q<2p$  and let $\Om:=\{x\in\mb R^N\; :\; f(x)>0\}$ is an open subset of $\mb R^N$ such that $0<\meas(\Om)<\infty$. Then there exists $\la^*>0$ such that for all $\la\in(0,\la^*)$, \eqref{eq1} admits a sequence of nontrivial weak solutions $\{u_k\}$ in $D^{1,p}(\mb R^N)\cap L^\infty(\mb R^N)$ with negative energy and $u_k\to 0$ strongly in $D^{1,p}(\mb R^N)$ as $k\to\infty$.
	\end{theorem}
	\begin{theorem}\label{sym-infinite-sol}
				Let 
				 $q=2p$. Then there exists  positive constants $\hat a$  such that for all $a>\hat a$ and for all $\la\in(0,\,aS\|f\|^{-1}_{\frac{p^*}{p^*-p}})$,  \eqref{eq1} has infinitely many nontrivial weak solutions in $D^{1,p}(\mb R^N)\cap L^\infty(\mb R^N).$
	\end{theorem}
%
%
	\begin{theorem}\label{dual-fount-sol}
	Let 
	$2p<q<2p^*$. Then for all $\la>0$, \eqref{eq1} has infinitely many solutions  in $D^{1,p}(\mb R^N)\cap L^\infty(\mb R^N)$.
	\end{theorem}
	\begin{remark}
	  We would also like to highlight that our  results are also new  for the classical $p$-Kirchhoff equation:
	\begin{equation}\label{eq2}  \quad \left\{
		\begin{array}{lr}
			\left(a+b\ds\int_{\mathbb R^N}|\nabla u|^{p}dx\right)\De_{p} u= \la f(x) |u(x)|^{q-2} u(x)+\left(\ds\int_{\mathbb R^N}\frac{|u(y)|^{p_{\beta,\mu}^{*}}}{|x-y|^{\mu}|y|^{\beta}}dy\right)\frac {|u(x)|^{p_{\beta,\mu}^{*}-2} u(x)}{|x|^\beta}  \; \text{in}\; \mathbb R^N
		\end{array}
		\right.
	\end{equation}  
	where all the parameters satisfy \eqref{assumption1}, $p<q<p^*$ and $f(\geq 0)\in L^{\frac{p^*}{p^*-q}}(\mathbb R^N)$. 
	\end{remark}
{\bf Notations}
\begin{itemize}
\item The constants $K,\;C$ and $ C_i$, $1=1,2,3,\cdots$ are positive real numbers (only depend on $N,p,\beta,\mu,q,a,b$, if nothing is mentioned) which may vary from line to line.
\item For any real number $r>0$, $B_r(0)$ denoted the ball of radius $r$ centered at $0$ with respect to the norm topology in $D^{1,p}(\mb R^N)$.
\item For any set $S$, the closer of the set is denoted by $\overline S$.
\item If $A$ is a measurable set in $\mathbb{R}^{N}$, we the Lebesgue measure of $A$  by $\meas( A )$. 
\item The arrows $\rightharpoonup $, $\to $ denote weak convergence,  strong convergence, respectively.
\item The notation $o_n(1)$ means as $n\to\infty$, $o_n(1)\to0.$
\end{itemize}

%
%

	\section{\bf Preliminaries and variational structure}\label{sec2} 
First, for any $1<p<\infty$, we recall the definition of the Sobolev space  \[D^{1,p}(\mathbb R^N)=\left\{u\in L^{p^*}(\mb R^N): \int_{\mb R^N}|\nabla u|^p dx <\infty\right\},\] which is equipped with the norm \[\|u\| :=\|u\|_{D^{1,p}(\mathbb R^N)}=\left(\int_{\mb R^N}|\nabla u|^p dx\right)^{\frac 1p}. \] 
Here $p^*$ denotes the Sobolev critical exponent and $p^*=\frac{N}{N-p}$, if $p<N$ and $p^*=\infty$ if $N\geq p$. By $(D^{1,p}(\mb R^N))^*$ we denote the dual of $D^{1,p}(\mb R^N)$ and by $\langle\cdot,\cdot\rangle$ we denote the dual paring between   $D^{1,p}(\mb R^N)$ and its dual $(D^{1,p}(\mb R^N))^*.$ Concerning the doubly weighted Hardy-Littlewood-Sobolev inequality \eqref{HLSineq}, if \(\alpha=\ga=\ba\) and \(s=r\), then the integral is
\[\int_{\mathbb R^N}\int_{\mathbb R^N} \frac{|u(x)|^{t}|u(y)|^{t}}{|x|^{\beta}|x-y|^{\mu}|y|^{\beta}} dx dy\]
is well-defined if \(|u|^t  \in L^{q}(\mathbb R^N)\) for some \(q >1\) satisfying \(\frac{2}{q} + \frac{2\beta +\mu}{N}= 2\).

For \(u \in D^{1,p}(\mathbb R^N)\), by the Sobolev embedding theorem, we have \(p\leq tq \leq \frac{Np}{N-p}\). Thus
\[\frac{p(2N- 2\ba-\mu)}{2N} \leq t \leq \frac{p(2N -2\ba -\mu)}{2(N-p)}\]
In this sense, we call \({p_{*}}_{\ba,\mu} = \frac{p(2N- 2\beta-\mu)}{2N}\) the lower critical exponent and \(p^{*}_{\ba, \mu} = \frac{p(2N -2\ba -\mu)}{2(N-p)}\) the upper critical exponent in the sense of the weighted-Littlewood-Sobolev inequality. Also, we have $0<p_{\beta,\mu}^*<p^*<2p_{\beta,\mu}^*.$ Generally, for \(\ga=\ba \geq 0\) and \(2\ba+ \mu \leq N\), the limit embedding for the upper critical exponent leads to the inequality
\[\left(\int_{\mathbb R^N}\int_{\mathbb R^N} \frac{|u(x)|^{p_{\ba,\mu}^{*}}|u(y)|^{p_{\ba,\mu}^{*}}}{|x|^{\beta}|x-y|^{\mu}|y|^{\beta}} dx dy \right)^{\frac{1}{p^{*}_{\ba,\mu}}} \leq C\int_{\mb R^N} |\nabla u|^{p} dx.\]

We define  the following norm
\[\|u\|_{\ba,\mu} = \left( \int_{\mathbb R^N}\int_{\mathbb R^N} \frac{|u(x)|^{p_{\ba, \mu}^{*}}|u(y)|^{p_{\ba,\mu}^{*}}}{|x|^{\ba}|x-y|^{\mu}|y|^{\ba}} dx dy\right) ^{\frac {1}{2.p^{*}_{\ba,\mu}}} .\]
Let us set
\begin{align}\label{bc1}S_{\ba,\mu}:= \inf_{u\in D^{1,p}(\mb R^N)\setminus \{0\}} \frac{\displaystyle\int_{\mathbb R^N} |\nabla u|^p dx}{ \left( \displaystyle\int_{\mathbb R^N}\int_{\mathbb R^N} \frac{|u(x)|^{p_{\ba, \mu}^{*}}|u(y)|^{p_{\ba,\mu}^{*}}}{|x|^{\ba}|x-y|^{\mu}|y|^{\ba}} dx dy\right) ^{\frac {p}{2.p^{*}_{\ba,\mu}}}}.
\end{align}

From the weighted Hardy-Littlewood-Sobolev inequality \eqref{HLSineq}, for all \(u \in D^{1,p}(\mathbb R^N),\) we know
\[\|u\|_{\ba,\mu}^2 \leq C(N,\ba, \mu)^{\frac{1}{p_{\ba, \mu}^*}} \|u\|_{p^*}^{p}. \]
Then
\[S_{\ba,\mu}\geq \frac{S}{C(N,\ba, \alpha)^{\frac{1}{p_{\ba, \mu}^*}}} >0,\]
where \(S\) is the best Sobolev constant for the embedding \(D^{1,p}(\mathbb R^N)\) into \(L^{p^*}(\mathbb R^N)\) is defined as
\begin{align}\label{se}S= \inf_{u\in D^{1,p}(\mb R^N)\setminus \{0\},\|u\|_{p^*}=1}\left\{\int_{\mathbb R^N} |\nabla u|^p dx\right\}.
	\end{align}
Observe that the natural energy functional related to \eqref{eq1} is not well defined for $u\in D^{1,p}(\mb R^N)$. To overcome this difficulty, we employ the following change of variables which was introduced in \cite{CJ}, namely, $w:=g^{-1}(u),$ where $g$ is defined by
\begin{equation}\label{g}
	\left\{
	\begin{array}{l}
		g^{\prime}(s)=\displaystyle\frac{1}{\left(1+2^{p-1}|g(s)|^{p}\right)^{\frac{1}{p}}}~~\mbox{in}~~ [0,\infty),\\
		g(s)=-g(-s)~~\mbox{in}~~ (-\infty,0].
	\end{array}
	\right.
\end{equation}

\noi	Now we state  some important and useful properties of $g$. For the detailed proofs of such results, one can see \cite{CJ,do} and references there in.

\begin{lemma}\label{L1}
	The function $g$ satisfies the following properties:
	\begin{itemize}
		\item[$(g_1)$] $g$ is uniquely defined, $C^{\infty}$ and invertible;
		\item[$(g_2)$] $g(0)=0$;
		\item[$(g_3)$] $0<g^{\prime}(s)\leq 1$ for all $s\in \mathbb{R}$;
		\item[$(g_4)$] $\frac{1}{2}g(s)\leq sg^{\prime}(s)\leq g(s)$ for all $s>0$;
		\item[$(g_5)$] $|g(s)|\leq |s|$ for all $s\in \mathbb{R}$;
		\item[$(g_6)$] $|g(s)|\leq 2^{1/{(2p)}}|s|^{1/2}$ for all $s\in \mathbb{R}$;
		\item[$(g_7)$] $\ds \lim_{s\to+\infty}\frac {g(s)}{s^{\frac 12}}=2^{\frac{1}{2p}}$;
		\item[$(g_8)$]  $|g(s)|\geq g(1)|s|$ for $|s|\leq 1$ and $|g(s)|\geq g(1)|s|^{1/2}$ for $|s|\geq 1$;
		\item[$(g_9)$] $g^{\prime \prime}(s)<0$ when $s>0$ and $g^{\prime \prime}(s)>0$ when $s<0$.
		\item[$(g_{10})$] $\displaystyle\lim_{s\to 0}\frac{g(s)}{s}=1.$ 
		\item[$(g_{11})$] $|g(s)g^{\prime}(s)|<\frac{1}{2^{\frac{p-1}{p}}}$ for all $s\in \mathbb{R}$;
		\item[$(g_{12})$] $g''(s)=-2^{p-1}(g(s))^{p-1}(g'(s))^{2+p},$ $s>0.$
	\end{itemize}
\end{lemma}
\noi	After applying the change of variable  $w=g^{-1}(u),$ we define the new functional $\mc I_{\la}:D^{1,p}(\mb R^N)\to \mb R$ as
\begin{align}\label{energy} 
	\mc I_{\la}(w)&=\frac{a}{p}\displaystyle\int_{{\mathbb R^N}} |\nabla w|^{p}dx+\frac{ b}{2p}\left( \int_{\mathbb R^N}|g^{\prime}(w)|^p|\nabla w|^pdx\right)^2\notag\\ &\qquad\quad-\frac{\la}{q}\int_{\mathbb R^N}f(x)|g(w)|^q dx -\frac {1}{4p_{\beta,\mu}^*}\int_{{\mathbb R^N}}\left(\int_{{\mathbb R^N}} \frac{|g(w(y))|^{2p_{\beta,\mu}^*}}{ |y|^\beta|x-y|^\mu}dy\right)\frac{|g(w(x))|^{2p_{\beta,\mu}^*}}{|x|^\beta}dx.
\end{align}
Note that, if   $w \in D^{1,p}(\mb R^N)$ is a critical point of the functional $\mc I_{\la}$, then for every $v \in D^{1,p}(\mb R^N)$, $\langle \mc I_{\la}^{\prime}(w), v\rangle =0$. That is,
{\begin{align}\label{2.5}
		&a\int_{{\mathbb R^N}} |\nabla w|^{p-2}\nabla w\nabla v dx+ b\int_{\mathbb R^N}|g^{\prime}(w)|^p|\nabla w|^p dx \int_{\mb R^N}\left(|g^{\prime}(w)|^p|\nabla w|^{p-2}\nabla w\nabla v + |g^{\prime}(w)|^{p-2} g^{\prime}(w) g^{\prime\prime}(w)|\nabla w|^p v \right) dx \notag\\
		&\qquad- \la\int_{\mb R^N}  f(x) |g(w)|^{q-2} g(w) g^{\prime}(w) v dx-\int_{{\mathbb R^N}} \left(\ds\int_{\mathbb R^N}\frac{|g(w)|^{2p_{\beta,\mu}^{*}}}{|x-y|
			^{\mu}|y|^{\beta}}dy\right)\frac{|g(w)|^{2p_{\beta,\mu}^{*}-2} g(w)}{|x|^{\beta}} g^{\prime}(w) v dx  =0
\end{align}} 
and $w$ is a weak solution to the following problem:
\begin{align}\label{2p}
		&- a\Delta_p w-b\int_{\mb R^N}|g^{\prime}(w)|^p|\nabla w|^p dx\cdot \left(|g^{\prime}(w)|^{p-2}g^{\prime}(w)g^{\prime\prime}(w)|\nabla w|^p+|g^{\prime}(w)|^p{div(|\nabla w|^{p-2}\nabla w) }\right)  \notag\\ &\qquad = \la f(x) |g(w)|^{q-2} g(w) g^{\prime}(w)+ \left(\ds\int_{\mathbb R^N}\frac{|g(w)|^{2p_{\beta,\mu}^{*}}}{|x-y|^{\mu}|y|^{\beta}}dy\right)\frac{|g(w)|^{2p_{\beta,\mu}^{*}-2} g(w)}{|x|^{\beta}} g^{\prime}(w)   \; \text{\;\; in \;}\; \mathbb R^N
\end{align}
See  that the transformed problem \eqref{2p} is equivalent to \eqref{eq1} which takes $u = g(w)$ as its solutions. Thus, now our aim is to find the  solutions to  \eqref{2p}.

\section{\bf Compactness arguments}
\noi In this section, first we recall the definition of Palais-Smale sequence.
\begin{definition}
	Let $\mc J: X\rightarrow \mathbb R$ be a $C^{1}$ functional on a Banach space $X$.
	\begin{enumerate}
		\item For $c\in \mb R$, a sequence $\{u_n\}\subset X$ is a Palais-Smale sequence at level $c$ (in short $(PS)_c$) in $X$ for $\mc J$ if $\mc J(u_n)=c +o_{n}(1)$ and $\mc J^{\prime}(u_n) \rightarrow 0$ in $X^{*}$(dual of $X$) as $k\rightarrow \infty.$
		\item We say $\mc J$ satisfies $(PS)_{c}$ condition if for any Palais-Smale sequence $\{u_n\}$ in $X$ for $\mc J$ has a convergent subsequence.
	\end{enumerate}
\end{definition}
\noi Next, we state the following concentration compactness Lemms due to Lions \cite{pl1}.



\begin{lemma}\label{y1}
	Let $\{u_n\}$ be a bounded sequence in $D^{1, p}(\mathbb R^N)$ converging weakly and a.e. to $u\in D^{1, p}(\mathbb R^N)$ such that $|\na u_n|^p\rightharpoonup \nu$, $|u_n|^{p^*}\rightharpoonup \omega$ in the sense of measure. Then, for at most countable set $J$, there exist families of distinct points $\{\nu_j: j\in J\}$ and $\{\omega_j: j\in J\}$ in $\mathbb R^N$ satisfying
	\begin{align*}
		\nu &\geq  |\na u|^p + \sum_{i \in J} \nu_j \delta_{z_j},\;\nu_j>0,\\
		\omega &= |u|^{p^*} + \sum_{i \in J} \omega_j \delta_{z_j},\; \omega_j>0,\\
		& S \omega_j^{\frac{p}{p^*}} \leq  \nu_j,
	\end{align*}
	where $\nu$, $\omega$ are bounded and nonnegative measures on $\mathbb R^N$ and $\delta_{z_j}$ is the Dirac mass at $z_j$.
\end{lemma}

\begin{lemma}\label{y2}
	Let \(\{u_n\}\subset  D^{1,p}(\mathbb R^N)\) be a sequence in  as in Lemma \ref{y1} and defined
	\[ \nu_{\infty}:=\lim_{R\ra\infty}\lim_{ n\ra\infty} \int_{|x|>R} |\na u_n|^p dx,\quad  \omega_{\infty}= \lim_{R\ra\infty}\lim_{ n\ra\infty} \int_{|x|>R} |u_n|^{p^*} dx\]
	Then it follows that
	\[S \omega_{\infty}^{p/p^*} \leq \nu_{\infty}\]
	and
	\[\limsup_{n\ra \infty} \int_{\mathbb R^N} |\na u_n|^p dx = \nu_{\infty} +\int_{\mathbb R^N} d\nu, \quad \limsup_{n\ra \infty} \int_{\mathbb R^N} |u_n|^{p^*} dx = \omega_{\infty} +\int_{\mathbb R^N} d\omega\]
\end{lemma}

Now we prove the following concentration-compactness lemma related to our problem.
\begin{lemma}\label{c1}
Let  $\beta\geq 0,$ $\mu>0,$ $0<2\beta+\mu<N$ and  $2\leq p<N$. If \(\{u_n\}\) is a bounded sequence in \(D^{1,p}(\mathbb  R^N)\) converges weakly and a.e in $\mb R^N$ to some \(u\in D^{1,p}(\mb R^N)\) as \(n\ra \infty\) and such that \(|u_n|^{p^*} \rightharpoonup \omega \) and \(|\na u_n|^p \rightharpoonup \nu\) in the sense of measure. Assume that
	\[\left(\int_{\mathbb R^N}\frac{|u_{n}(y)|^{p^{*}_{\ba,\mu}}}{|y|^{\ba}|x-y|^{\mu}}\right)\frac{|u_{n}(x)|^{p^{*}_{\ba,\mu}}}{|x|^{\ba}} \rightharpoonup \zeta\]
	weakly in the sense of measure, where \(\zeta\) is a bounded positive measure on \(\mathbb R^N\) and define
	\[\nu_{\infty} := \lim_{R\ra \infty} \limsup_{n\ra \infty} \int_{|x|\geq R} |\na u_n|^p dx, \quad \omega_{\infty}:= \lim_{R\ra \infty} \limsup_{n\ra \infty} \int_{|x|\geq R} |u_n|^{p^*} dx. \]
	\[ \zeta_{\infty} := \lim_{R\ra \infty} \limsup_{n\ra \infty} \int_{|x|> R}\left(\int_{\mathbb R^N}\frac{|u_{n}(y)|^{p^{*}_{\ba,\mu}}}{|y|^{\ba}|x-y|^{\mu}}dy\right)\frac{|u_{n}(x)|^{p^{*}_{\ba,\mu}}}{|x|^{\ba}} dx.\]
	Then there exists a countable sequence of points \(\{z_j\}_{j\in J}\subset\mathbb R^N\) and families of positive numbers $\{\omega_j : j\in J\}$, $\{\zeta_j: j\in J\}$ and $\{\nu_j: j\in I\}$ such that
	\begin{align}
		\zeta & = \left(\int_{\mathbb R^N}\frac{|u(y)|^{p^{*}_{\ba,\mu}}}{|y|^{\ba}|x-y|^{\mu}}\right) \frac{|u(x)|^{p^{*}_{\ba,\mu}}}{|x|^{\ba}} +\sum_{j\in J} \zeta_{j} \delta_{z_j},\quad \sum_{j\in J} \zeta_j^{\frac{1}{p^{*}_{\ba,\mu}}} < \infty,\label{l11}\\
		\nu &\geq |\na u|^p + \sum_{j\in J} \nu_{j} \delta_{z_j},\label{l2}\\
		\omega &\geq  |u|^{p^{*}} + \sum_{j\in J} \omega_{j} \delta_{z_j},\label{l3}
	\end{align}
	and
	\begin{align}
		S_{\ba, \mu} \zeta_{j}^{\frac{p}{2 p^{*}_{\ba,\mu}}}\leq \nu_j, \;\;\;\mbox{and} \; \;\;\zeta_{j}^{\frac{N}{2N-2\ba- \mu}} \leq C(N,\ba, \mu)^{\frac{N}{2N-2\ba- \mu}} \omega_j.\label{l4}
	\end{align}
	where \(\delta_z\) is the Dirac-mass of mass \(1\) concentrated at \(z \in\mathbb R^N\).
	
	\noi	For the energy at infinity, we have
	\[\limsup_{n\ra \infty} \int_{\mathbb R^N } |\na u_n|^p dx= \nu_{\infty} + \int_{\mathbb R^N} d\nu, \quad \limsup_{n\ra \infty} \int_{\mathbb R^N} |u_n|^{p^*} dx = \omega_{\infty} + \int_{\mathbb R^N} d\omega \]
	\[\limsup_{n\ra \infty} \int_{\mathbb R^N}\int_{\mathbb R^N}\frac{|u(y)|^{p^{*}_{\ba,\mu}}|u(x)|^{p^{*}_{\ba,\mu}}}{|x|^{\ba}|x-y|^{\mu}|y|^{\ba}}   dxdy = \zeta_{\infty}+ \int_{\mathbb R^N} d\zeta,\]
	and
	\[C(N,\ba,\mu)^{-\frac{2N}{2N-2\ba-\mu}} \zeta_{\infty}^{\frac{2N}{2N-2\ba-\mu}} \leq \omega_{\infty} \left(\int_{\mathbb R^N} d\omega + \omega_{\infty}\right), \]
	\[S^{p} C(N, \ba,\mu)^{-\frac{p}{p_{\ba,\mu}^{*}}} \zeta_{\infty}^{\frac{p}{p_{\ba,\mu}^{*}}} \leq \nu_{\infty} \left(\int_{\mathbb R^N} d\nu + \nu_{\infty}\right).\]
\end{lemma}

\begin{proof}
	Let $v_n:=u_n-u$. Then the sequence $\{v_n\}$ converges weakly to $0$ in $D^{1, p}(\mathbb R^N)$ and $v_n(x)\rightarrow 0$ a.e. in $\mathbb R^N$ as the bounded sequence $\{u_n\}$ converges weakly to $u$ in $D^{1, p}(\mathbb R^N)$. By Lemmas \ref{y1}- \ref{y2}, we have
	\begin{align*}
		|\na v_n|^p\rightharpoonup &\; \tau_1:=\nu-|\na u|^p,\\
		|v_n|^{p^*}\rightharpoonup &\; \tau_2:=\omega-|u|^{p^*},\\
		\left(\int_{\mathbb R^N}\frac{|v_n(y)|^{p_{\ba, \mu}^*}}{|y|^{\ba}|x-y|^{\mu}}dy\right)\frac{|v_n(x)|^{p_{\ba, \mu}^*}}{|x|^{\ba}}\rightharpoonup &\;\tau_3:=\zeta - \left(\int_{\mathbb R^N}\frac{|u(y)|^{p_{\ba, \mu}^*}}{|y|^{\ba}|x-y|^{\mu}}dy\right)\frac{|u(x)|^{p_{\ba, \mu}^*}}{|x|^{\ba}}.
	\end{align*}
	Firstly, we show that for every $\phi \in C_c^{\infty}(\mathbb R^N)$,
	\begin{align}\label{n32}
		\left|\int_{\mathbb R^N}	\left(\int_{\mathbb R^N}\frac{|\phi v_n(y)|^{p_{\ba,\mu}^*}}{|y|^{\ba}|x-y|^{\mu}}dy\right)\frac{|\phi v_n(x)|^{p_{\ba,\mu}^*}}{|x|^{\ba}}-	\int_{\mb R^N}\left(\int_{\mathbb R^N}\frac{|v_n(y)|^{p_{\ba,\mu}^*}}{|y|^{\ba}|x-y|^{\mu}}dy\right)\frac{|\phi(x)|^{p_{\ba,\mu}^*}|\phi v_n(x)|^{p_{\ba,\mu}^*}}{|x|^{\ba}}\right|\rightarrow 0,
	\end{align} as $n\to\infty.$
	For this, we denote \[\Psi_n(x):=\left(\int_{\mathbb R^N}\frac{|\phi v_n(y)|^{p_{\ba,\mu}^*}}{|y|^{\ba}|x-y|^{\mu}}dy\right)\frac{|\phi v_n(x)|^{p_{\ba,\mu}^*}}{|x|^{\ba}}-	\left(\int_{\mathbb R^N}\frac{|v_n(y)|^{p_{\ba,\mu}^*}}{|y|^{\ba}|x-y|^{\mu}}dy\right)\frac{|\phi(x)|^{p_{\ba, \mu}^*} |\phi v_n(x)|^{p_{\ba,\mu}^*}}{|x|^{\ba}}.\] 
	As $\phi \in C_c^{\infty}(\mathbb R^N)$, so we have for every $\delta >0$, there exists $K>0$ such that
	\begin{align}\label{n31}
		\int_{|x|\geq K}|\Psi_n(x)|dx < \delta,\;\; \text { for all}\;n\geq 1.
	\end{align}
	Further, we know that the Riesz potential defines a linear operator and $v_n(x)\rightarrow 0$ a.e. in $\mathbb R^N$ and hence, 
	\begin{align*}
		\int_{\mathbb R^N}\frac{|v_n(y)|^{p_{\ba,\mu}^*}}{|y|^{\ba}|x-y|^{\mu}}dy\rightarrow 0\;\text{a.e. in}\;\mathbb R^N.
	\end{align*}
	Thus, we have $\Psi_n(x)\rightarrow 0$ a.e. in $\mathbb R^N$. We note that
	\begin{align*}
		\Psi_n(x)
		=& \left(\int_{\mathbb R^N}\frac{\left(|\phi(y)|^{p_{\ba,\mu}^*}-|\phi(x)|^{p_{\ba, \mu}^*}\right)|v_n(y)|^{p_{\ba, \mu}^*}}{|y|^{\ba}|x-y|^{\mu}}dy \right)\frac{|\phi v_n(x)|^{p_{\ba,\mu}^*}}{|x|^{\ba}}
		:=&\left(\int_{\mathbb R^N}\Phi(x, y)|v_n(y)|^{p_{\ba, \mu}^*}dy\right) \frac{|\phi v_n(x)|^{p_{\ba, \mu}^*}}{|x|^{\ba}},
	\end{align*}
	where $\Phi(x, y)=\frac{|\phi(y)|^{p_{\ba,\mu}^*}-|\phi(x)|^{p_{\ba,\mu}^*}}{|y|^{\ba}|x-y|^{\mu}}$. Moreover, {for almost all} $x\in\mb R^N$, there exists some $R>0$ large enough such that
	\begin{align*}
		\int_{\mathbb R^N}\Phi(x, y)|v_n(y)|^{p_{\ba,\mu}^*}dy
		=&\int_{|y|\leq R}\Phi(x, y)|v_n(y)|^{p_{\ba, \mu}^*}dy - |\phi(x)|^{p_{\ba,\mu}^*}\int_{|y|\geq R}\frac{|v_n(y)|^{p_{\ba,\mu}^*}}{|y|^{\ba}|x-y|^{\mu}}dy.
	\end{align*}
	In \cite{pl1}, we noticed that $\Phi(x, y)\in L^{r}(B_R)$ for each $x$, where $r<\frac{N}{\ba+\mu}$ if $\mu>1$, $r\leq \frac{N}{\mu}$ if $0<\mu \leq 1$. So, by Young's inequality, there exists $t>\frac{2N}{\mu}$ such that
	\begin{align*}
		\left(\int_{B_K(0)}\left(\int_{B_R(0)}\Phi(x, y)|v_n(y)|^{p_{\ba,\mu}^*}dy\right)^t dx\right)^{\frac 1t}\leq L_{\phi}\|\Phi(x, y)\|_{r} \||v_n|^{p_{\ba,\mu}^*}\|_{\frac{2N}{2N-2\ba-\mu}}\leq L^{\prime}_{\phi},
	\end{align*}
	where $K$ is same as in $\eqref{n31}$ and $L_\phi^\prime$ is some positive constant that depends on $\phi$. Moreover, one can easily see that for $R>0$ large enough
	\begin{align*}
		\left(\int_{B_K(0)}\left(|\phi(x)|^{p_{\ba, \mu}^*}\int_{|y|\geq R}\frac{|v_n(y)|^{p_{\ba,\mu}^*}}{|y|^{\ba}|x-y|^{\mu}}dy\right)^t dx\right)^{\frac 1t}\leq L,
	\end{align*}
	and so, we have
	\begin{align*}
		\left(\int_{B_K(0)}\left(\int_{\mathbb R^N}\Phi(x, y)|v_n(y)|^{p_{\ba,\mu}^*}dy\right)^t dx\right)^{\frac 1t}\leq L^{\prime\prime}_{\phi},
	\end{align*} where  $L$ and $L_\phi^{\prime\prime}$( depends on $\phi$) are some positive constants.
	Thus for $s=\frac{t\mu-2N}{2(2N+2Nt-t\mu)} >0$ small enough, we obtain
	\begin{align*}
		\int_{B_K(0)}|\Psi_n(x)|^{1+s}dx&\leq\left(\int_{B_K(0)}\left(\int_{\mathbb R^N}\Phi(x, y)|v_n(y)|^{p_{\ba,\mu}^*}dy\right)^t dx\right)^{\frac {1+s}{t}}\left(\int_{B_K(0)}|\phi v_n|^{p^*}dx\right)^{\frac{p_{\ba, \mu}^* (1+s)}{p^*}} \\
		&\quad\quad\left(\int_{B_K(0)} \frac{1}{|x|^{\frac{2Nt\ba(1+s)}{2Nt-2N(1+s)-t(2N-2\ba-\mu)(1+s)}}}\right)^{\frac{2Nt-2N(1+s)-t(2N-2\ba-\mu)(1+s)}{2Nt}}\leq L^{\prime\prime}_{\phi},
	\end{align*}
	since $\frac{2Nt\ba(1+s)}{2Nt-2N(1+s)-t(2N-2\ba-\mu)(1+s)}<N.$	Using this together with $\Psi_n(x)\rightarrow 0$ a.e. in $\mathbb R^N$, we achieve
	\begin{align*}
		\int_{B_K(0)}|\Psi_n(x)|dx \rightarrow 0\;\text{as}\;n\rightarrow\infty.
	\end{align*}
	Combining this with \eqref{n31}, we infer that
	\begin{align*}
		\int_{\mathbb R^N}|\Psi_n(x)|dx \rightarrow 0\;\text{as}\;n\rightarrow\infty.
	\end{align*}
	Now for every $\phi \in C_c^{\infty}(\mathbb R^N)$, by the weighted Hardy-Littlewood-Sobolev inequality \eqref{HLSineq}, we deduce
	\begin{align*}
		\int_{\mathbb R^N}\left(\int_{\mathbb R^N}\frac{|\phi v_n(y)|^{p_{\ba, \mu}^*}}{|y|^{\ba}|x-y|^{\mu}}dy\right) \frac{|\phi v_n(x)|^{p_{\ba, \mu}^*}}{|x|^{\ba}} dx \leq C(N, \ba,\mu)\|\phi v_n\|_{p^*}^{2p_{\ba, \mu}^*}.
	\end{align*}
	Thus, the equation \eqref{n32} is proved.
	From equation \eqref{n32}, we get
	\begin{align*}
		\int_{\mathbb R^N}|\phi(x)|^{{2p_{\ba, \mu}^*}}\left(\int_{\mathbb R^N}\frac{|v_n(y)|^{p_{\ba, \mu}^*}}{|y|^{\ba}|x-y|^{\mu}}dy\right)\frac{|v_n(x)|^{p_{\ba, \mu}^*}}{|x|^{\ba}} dx \leq C(N, \ba,\mu)\|\phi v_n\|_{p^*}^{2p_{\ba, \mu}^*}+o_n(1).
	\end{align*}
	On taking the limit as $n\rightarrow \infty$, we obtain
	\begin{align}\label{n33}
		\int_{\mathbb R^N}|\phi(x)|^{2p_{\ba, \mu}^*}d\tau_3\leq C(N, \ba, \mu)\left(\int_{\mathbb R^N}|\phi|^{p^*}d\tau_2\right)^{\frac{2p_{\ba, \mu}^*}{p^*}}.
	\end{align}
	\noi	Further, let $\phi=\chi_{\{z_j\}}$, $j\in J$ and using this in \eqref{n33}, we have
	\begin{align*}
		\zeta_j^{\frac{p^*}{2p_{\ba, \mu}^*}}\leq \left(C(N, \ba,\mu)\right)^{\frac{p^*}{2p_{\ba, \mu}^*}} \omega_j,\;\text { for all } \;j\in J.
	\end{align*}
	Now the definition of $S_{\ba, \mu}$ (see \eqref{bc1}) yields that
	\begin{align*}
		\left(\int_{\mathbb R^N}\left(\int_{\mathbb R^N}\frac{|\phi v_n(y)|^{p_{\ba, \mu}^*}}{|y|^{\ba}|x-y|^{\mu}}dy\right)\frac{|\phi v_n(x)|^{p_{\ba, \mu}^*}}{|x|^{\ba}}dx\right)^{\frac{p}{2p_{\ba, \mu}^*}}S_{\ba,\mu}\leq \int_{\mathbb R^N}|\na(\phi v_n)|^p dx.
	\end{align*}
	Also, using \eqref{n32} and $v_n\rightarrow 0$ in $L^p_{loc}(\mathbb R^N)$, it follows that
	\begin{align*}
		\left(\int_{\mathbb R^N}|\phi(x)|^{2p_{\ba, \mu}^*}\left(\int_{\mathbb R^N}\frac{|v_n(y)|^{p_{\ba, \mu}^*}}{|y|^{\ba}|x-y|^{\mu}}dy\right) \frac{| v_n(x)|^{p_{\ba, \mu}^*}}{|x|^{\ba}}dx\right)^{\frac{p}{2p_{\ba, \mu}^*}}S_{\ba,\mu}\leq \int_{\mathbb R^N}\phi^p|\na v_n|^p dx + o_n(1).
	\end{align*}
	On passing the limit as $n \rightarrow \infty$ in the above estimation, we achieve
	\begin{align}\label{n34}
		\left(\int_{\mathbb R^N}|\phi(x)|^{2p_{\ba, \mu}^*}d \tau_3\right)^{\frac{p}{2p_{\ba, \mu}^*}}S_{\ba,\mu}\leq \int_{\mathbb R^N}|\phi|^p d\tau_1.
	\end{align}
	Let $\phi=\chi_{\{z_j\}}$, $j\in J$ and applying this in \eqref{n34}, we have
	\begin{align*}
		S_{\ba,\mu}\zeta_j^{\frac{p}{2p_{\ba, \mu}^*}}\leq \nu_j,\;\forall \;j\in J.
	\end{align*}
	This completes the proof of \eqref{l4}.
	
	Now, we prove the possible loss of mass at infinity. For $R>1$, let $\phi_{R}\in C^{\infty}(\mathbb R^N)$ be such that $\phi_{R}=1$ for $|x|> R+1$, $\phi_{R}(x)=0$ for $|x|<R$ and $0\leq \phi_{R}(x) \leq 1$ on $\mathbb R^N$. For every $R>1$, we have
	\begin{align*}
		\limsup_{n\ra \infty} & \int_{\mathbb R^N} \int_{\mathbb R^N}\frac{|u_{n}(y)|^{p^{*}_{\ba, \mu}}|u_{n}(x)|^{p^{*}_{\ba, \mu}}}{|x|^{\ba}|x-y|^{\mu} |y|^{\ba}} dy dx\\
		&= 	\limsup_{n\ra \infty} \left( \int_{\mathbb R^N} \int_{\mathbb R^N}\frac{|u_{n}(y)|^{p^{*}_{\ba, \mu}}|u_{n}(x)|^{p^{*}_{\ba, \mu}}\phi_{R}(x)}{|x|^{\ba}|x-y|^{\mu} |y|^{\ba}} dy dx + 	 \int_{\mathbb R^N} \int_{\mathbb R^N}\frac{|u_{n}(y)|^{p^{*}_{\ba, \mu}}|u_{n}(x)|^{p^{*}_{\ba, \mu}}(1-\phi_{R}(x))}{|x|^{\ba}|x-y|^{\mu} |y|^{\ba}} dy dx  \right)\\
		&= 	\limsup_{n\ra \infty} \int_{\mathbb R^N} \int_{\mathbb R^N}\frac{|u_{n}(y)|^{p^{*}_{\ba, \mu}}|u_{n}(x)|^{p^{*}_{\ba, \mu}}\phi_{R}(x)}{|x|^{\ba}|x-y|^{\mu} |y|^{\ba}} dy dx + 	\int_{\mathbb R^N} (1- \phi_{R}) d\zeta.
	\end{align*}
	Letting $R\ra \infty$, by Lebesgue's  dominated convergent theorem, we deduce
	\[	\limsup_{n\ra \infty} \int_{\mathbb R^N} \int_{\mathbb R^N}\frac{|u_{n}(y)|^{p^{*}_{\ba, \mu}}|u_{n}(x)|^{p^{*}_{\ba, \mu}} }{|x|^{\ba}|x-y|^{\mu} |y|^{\ba}}  dy dx= \zeta_\infty +\int_{\mb R^N} d\zeta. \]
	By the weighted Hardy-Littlewood-Sobolev inequatlity \eqref{HLSineq}, we get
	\begin{align*}
		\zeta_{\infty} &=\lim_{R\ra\infty} \limsup_{n\ra \infty} \int_{\mathbb R^N} \left(\int_{\mathbb R^N}\frac{|u_n(y)|^{p^{*}_{\ba,\mu}}}{|x-y|^{\mu}|y|^{\ba}} dy \right) \frac{|\phi_{R}u_n(x)|^{p^{*}_{\ba,\mu}}}{|x|^{\ba}}  dx \\
		&\leq C(N, \ba,\mu) \lim_{R\ra\infty} \limsup_{n\ra \infty}  \left(\int_{\mathbb R^N} |u_n|^{p^*} dx \int_{\mathbb R^N} |\phi_{R} u_n|^{p^*} dx \right)^{\frac{2N-2\ba-\mu}{2N}}\\
		&=  C(N, \ba,\mu)  \left((\omega_{\infty}+ \int_{\mathbb R^N} d\omega) \omega_{\infty} \right)^{\frac{2N-2\ba-\mu}{2N}}.
	\end{align*}
	This gives 
	\[C(N,\ba, \mu)^{-\frac{2N}{2N-2\ba-\mu}} \zeta_{\infty}^{\frac{2N}{2N-2\ba-\mu}} \leq \omega_{\infty} \left(\int_{\mathbb R^N} d\omega + \omega_{\infty}\right).\]
	Similarly, using the weighted Hardy-Littlewood-Sobolev inequatlity \eqref{HLSineq}, we obtain
	\begin{align*}
		\zeta_{\infty} &=\lim_{R\ra\infty} \limsup_{n\ra \infty} \int_{\mathbb R^N} \left(\int_{\mathbb R^N}\frac{|u_n(y)|^{p^{*}_{\ba,\mu}}}{|x-y|^{\mu}|y|^{\ba}} dy \right) \frac{|\phi_{R}u_n(x)|^{p^{*}_{\ba,\mu}}}{|x|^{\ba}}  dx \\
		&\leq C(N, \ba,\mu) \lim_{R\ra\infty} \limsup_{n\ra \infty}  \left(\int_{\mathbb R^N} |u_n|^{p^*} dx \int_{\mathbb R^N} |\phi_{R} u_n|^{p^*} dx \right)^{\frac{2N-2\ba-\mu}{2N}}\\
		&\leq C(N, \ba,\mu) S^{-p_{\ba,\mu}^{*}} \lim_{R\ra\infty} \limsup_{n\ra \infty}  \left(\int_{\mathbb R^N} |\na u_n|^{p} dx \int_{\mathbb R^N} |\na(\phi_{R} u_n)|^{p} dx \right)^{\frac{p^{*}_{\ba,\mu}}{p}}\\
		&=  C(N, \ba,\mu) S^{-p_{\ba,\mu}^{*}} \left((\nu_{\infty}+ \int_{\mathbb R^N} d\nu)\nu_{\infty} \right)^{\frac{p^{*}_{\ba,\mu}}{p}},
	\end{align*}
	which implies that
	\[S^{p} C(N, \ba,\mu)^{-\frac{p}{p_{\ba,\mu}^{*}}} \zeta_{\infty}^{\frac{p}{p_{\ba,\mu}^{*}}} \leq \nu_{\infty} \left(\int_{\mathbb R^N} d\nu + \nu_{\infty}\right).\]
	This completes the proof.	
\end{proof}

\begin{lemma}\label{lb}
 Assume that  \( 2< q < 2p\) and \eqref{assumption1} hold. Then any $(PS)_c$ sequence  for $\mc I_\la$ is bounded in \(D^{1,p}(\mathbb R^N)\).
\end{lemma}

\begin{proof}
	Let \(\{w_n\}\) be a $(PS)_c$ sequence in \(D^{1,p}(\mathbb R^N)\). Then
	\begin{align*}
		c+o_n(1)=\mc I_{\la}(w_n)&=\frac{a}{p}\displaystyle\int_{{\mathbb R^N}} |\nabla w_n|^{p}dx+\frac{ b}{2p}\left( \int_{\mathbb R^N}|g^{\prime}(w_n)|^p|\nabla w_n|^pdx\right)^2\\ 
		&\quad-\frac{\la}{q}\int_{\mathbb R^N}f(x)|g(w_n)|^q dx -\frac {1}{4p_{\beta,\mu}^*}\int_{{\mathbb R^N}}\left(\int_{{\mathbb R^N}} \frac{|g(w_{n}(y))|^{2p_{\beta,\mu}^*}}{ |y|^\beta|x-y|^\mu}dy\right)\frac{|g(w_{n}(x))|^{2p_{\beta,\mu}^*}}{|x|^\beta}dx.
	\end{align*}
	For any $v \in D^{1,p}(\mathbb R^N)$, we have
	\begin{align*}
		o_n(1)\|w_n\|=& \langle \mc I^{\prime}_{\la}(w_n), v\rangle\\
		=&  a\int_{{\mathbb R^N}} |\nabla w_n|^{p-2}\nabla w_n\nabla v dx- \la\int_{\mb R^N}  f(x) |g(w_n)|^{q-2} g(w) g^{\prime}(w_n) v dx\\
		& + b\int_{\mathbb R^N}|g^{\prime}(w_n)|^p|\nabla w_n|^p dx\int_{\mb R^N} \left(|g^{\prime}(w_n)|^p|\nabla w_n|^{p-2}\nabla w_n\nabla v dx + |g^{\prime}(w_n)|^{p-2} g^{\prime}(w_n) g^{\prime\prime}(w_n)|\nabla w_n|^p v \right) dx \\
		&\quad-\int_{{\mathbb R^N}} \left(\ds\int_{\mathbb R^N}\frac{|g(w_n)|^{2p_{\beta,\mu}^{*}}}{|x-y|
			^{\mu}|y|^{\beta}}dy\right)\frac{|g(w_n)|^{2p_{\beta,\mu}^{*}-2} g(w_n)}{|x|^{\beta}} g^{\prime}(w_n) v dx.
	\end{align*}
	Choose \(v_n= (1+2^{p-1} |g(w_n)|^{p})^{\frac{1}{p}} g(w_n)=\frac{g(w_n)}{g^{\prime}(w_n)} \in D^{1,p}(\mathbb R^N)\). Then Lemma \ref{L1}-\((g_4)\) and
	\[|\nabla v_n|=\left(1+\frac{2^{p-1} |g(w_n)|^{p}}{1+2^{p-1} |g(w_n)|^{p}}\right) |\nabla w_n|,\]
	yield, \(\|v_n\| \leq 2 \|w_n\|\).
	Also, by \(\eqref{2.5}\), we have
	\begin{align}\label{3.7}
		o_n(1)\|w_n\|&= \langle \mc I^{\prime}_{\la}(w_n), v_n\rangle = a \int_{\mathbb R^N} \left(1+\frac{2^{p-1} |g(w_n)|^{p}}{1+2^{p-1} |g(w_n)|^{p}}\right) |\nabla w_n|^p dx + b \left( \int_{\mathbb R^N} \frac{|\nabla w_n|^p}{1+2^{p-1} |g(w_n)|^{p}} dx \right)^2 \notag\\
		&\quad - \lambda \int_{\mathbb R^N} f(x) |g(w_n)|^{q} dx - \int_{\mathbb R^N}\int_{\mathbb R^N} \frac{|g(w_{n}(y))|^{2p_{\beta,\mu}^*} |g(w_{n}(x))|^{2p_{\beta,\mu}^*}}{ |y|^\beta|x-y|^\mu |x|^\beta} dy dx.
	\end{align}
	Now using \eqref{3.7} together with the H\"{o}lder inequality, Lemma \ref{L1}-$(g_4),(g_6)$ and \eqref{se}, we obtain
	\begin{align*}
		c + o_n(1) \|w_n\| & = \mc I_{\la}(w_n) - \frac{1}{4p^{*}_{\ba, \mu}}\langle \mc I^{\prime}_{\la}(w_n), v_n\rangle\\
		= &\; a \int_{\mathbb R^N} \left[\frac1p -\frac{1}{4p_{\ba,\mu}^{*}} \left(1+\frac{2^{p-1} |g(w_n)|^{p}}{1+2^{p-1} |g(v_n)|^{p}}\right)\right] |\nabla w_n|^p  + \left( \frac{1}{2p}-\frac{1}{4p_{\ba,\mu}^{*}}\right) b \left(\int_{\mathbb R^N} |g^{\prime}(w_n)|^p |\nabla w_n|^p \right)^2 \\
		&\quad \quad + \left(\frac{1}{4p_{\ba,\mu}^{*}} -\frac{1}{q}\right) \la\int_{\mathbb R^N} f(x)|g(w_n)|^{q} dx \\
		\geq &\;  a \left(\frac1p -\frac{1}{2p_{\ba,\mu}^{*}}\right) \int_{\mathbb R^N} |\nabla w_n|^p dx  - \left(\frac{1}{q}- \frac{1}{4p_{\ba,\mu}^{*}} \right) \la \|f\|_{\frac{2p^*}{2p^*-q}} \left(\int_{\mathbb R^N} |g(w_n)|^{2p^*} dx\right)^{\frac{q}{2p^*}} \\
		\geq &\;  a \left(\frac1p -\frac{1}{2p_{\ba,\mu}^{*}}\right) \int_{\mathbb R^N} |\nabla w_n|^p dx  - \la \left(\frac{1}{q}- \frac{1}{4p_{\ba,\mu}^{*}} \right)  2^{\frac{q}{2p}} S^{-\frac{q}{2p}} \|f\|_{\frac{2p^*}{2p^*-q}} \left(\int_{\mathbb R^N} |\na w_n|^{p} dx\right)^{\frac{q}{2p}} \\
		\geq & \left( \frac{1}{p}-\frac{1}{2p_{\ba,\mu}^{*}}\right) a \|w_n\|^{p} - \la\left(\frac{1}{q}-\frac{1}{4p_{\ba,\mu}^{*}} \right)  2^{\frac{q}{2p}} S^{-\frac{q}{2p}} \|f\|_{\frac{2p^*}{2p^*-q}} \|w_n\|^{\frac{q}{2}}.
	\end{align*}
	This implies  \(\{w_n\}\) is bounded, since $p<p_{\ba,\mu}^{*}$ and $2 <q< 2p$. \end{proof}
	\begin{lemma}\label{lbb}
		 Let \(  q = 2p\) and \eqref{assumption1} hold. Then any $(PS)_c$ sequence for $\mc I_\la$ is bounded in \(D^{1,p}(\mathbb R^N)\).
	\end{lemma}
\begin{proof}
Let  $\{w_n\}$ be a $(PS)_c$ sequence for for $\mc I_\la$ any $c\in \mb R^N.$	Using the similar calculation as in Lemma \ref{lb}, we get
	\begin{align*} 
		c + o(1) \|w_n\| & = \mc I_{\la}(w_n) - \frac{1}{4p^{*}_{\ba, \mu}}\langle \mc I^{\prime}_{\la}(w_n), v_n\rangle\\
		&\geq  \left( \frac{1}{p}-\frac{1}{2p_{\ba,\mu}^{*}}\right)\|w_n\|^{p}\left( a  - \frac \la 2  2 S^{-1} \|f\|_{\frac{p^*}{p^*-p}}\right).
	\end{align*} For all $0<\la<\frac{a}{ S^{-1} \|f\|_{\frac{p^*}{p^*-p}} },$ we get $\{w_n\}$ is a bounded sequence.
\end{proof}
\begin{lemma}\label{lbbb}
 Assume that \(  2p<q <2p^*_{\beta,\mu}\) and \eqref{assumption1} hold. Then any $(PS)_c$ sequence for $\mc I_\la$ is bounded in \(D^{1,p}(\mathbb R^N)\).
\end{lemma}
\begin{proof}Let  $\{w_n\}$ be a $(PS)_c$ sequence for $\mc I_\la$ for any $c\in \mb R^N.$	Gathering \eqref{3.7} in combination  with the H\"{o}lder inequality, Lemma \ref{L1}-$(g_4),(g_6)$ and \eqref{se}, it follows that
	\begin{align*}
		c + o_n(1) \|w_n\| & = \mc I_{\la}(w_n) - \frac{1}{q}\langle \mc I^{\prime}_{\la}(w_n), v_n\rangle\\
		= &\; a \int_{\mathbb R^N} \left[\frac1p -\frac{1}{q} \left(1+\frac{2^{p-1} |g(w_n)|^{p}}{1+2^{p-1} |g(v_n)|^{p}}\right)\right] |\nabla w_n|^p  + \left( \frac{1}{2p}-\frac{1}{q}\right) b \left(\int_{\mathbb R^N} |g^{\prime}(w_n)|^p |\nabla w_n|^p \right)^2 \\
		&\quad \quad + \left(\frac 1q-\frac {1}{4p_{\beta,\mu}^*}\right)\int_{{\mathbb R^N}}\left(\int_{{\mathbb R^N}} \frac{|g(w(y))|^{2p_{\beta,\mu}^*}}{ |y|^\beta|x-y|^\mu}dy\right)\frac{|g(w(x))|^{2p_{\beta,\mu}^*}}{|x|^\beta}dx\\
		\geq &\;  a \left(\frac1p -\frac{2}{q}\right) \int_{\mathbb R^N} |\nabla w_n|^p dx  
		= \left( \frac{1}{p}-\frac{2}{q}\right) a \|w_n\|^{p},
	\end{align*}where in the last line,  we used the fact that $2p<q<2p^*<4p_{\beta,\mu}^*$.
	Therefore, from the above estimation, it implies that  \(\{w_n\}\) is bounded.
	This completes the proof of the Lemma. 
\end{proof}
\begin{lemma}\label{sk}
Let  $\beta\geq 0,$ $\mu>0,$ $0<2\beta+\mu<N$ and  $2\leq p<N$. Suppose $\{w_n\}$ is a bounded sequence in $L^{p^*}(\mb R^N)$ such that $w_n\to w$ a.e. in $\mb R^N$. Then we have
\begin{align}\label{bll}
&	\int_{{\mathbb R^N}}\left(\int_{{\mathbb R^N}} \frac{|w_n|^{p_{\beta,\mu}^*}}{ |y|^\beta|x-y|^\mu}dy\right)\frac{|w_n|^{p_{\beta,\mu}^*}}{|x|^\beta}dx
	-\int_{{\mathbb R^N}}\left(\int_{{\mathbb R^N}} \frac{|w_n-w|^{p_{\beta,\mu}^*}}{ |y|^\beta|x-y|^\mu}dy\right)\frac{|w_n-w|^{p_{\beta,\mu}^*}}{|x|^\beta}dx\notag\\
	&\qquad\qquad\qquad\qquad\qquad\qquad\qquad\to \int_{{\mathbb R^N}}\left(\int_{{\mathbb R^N}} \frac{|w|^{p_{\beta,\mu}^*}}{ |y|^\beta|x-y|^\mu}dy\right)\frac{|w|^{p_{\beta,\mu}^*}}{|x|^\beta}dx
\end{align} as $n\to\infty$.
\end{lemma}
\begin{proof}
	The proof follows in a  similar manner as in \cite{my1}.
\end{proof}

\begin{lemma}\label{cc}
	 Assume that  \( 2\leq q < 2p\) and \eqref{assumption1} hold. Let  \(\{w_n\}\subset D^{1,p}(\mathbb R^N)\) be a Palais-Smale sequence for \(\mc I_{\la}\) and $c<0$,  then there exists $\la^*>0$ such that $\mc I_{\la}$ satisfies the $(PS)_c$ condition for all $\la\in (0,\la^*)$.
\end{lemma}

\begin{proof}
	Let $\{w_n\}\subset D^{1,p}(\mathbb R^N)$ be a $(PS)_c$-sequence for $\mc I_\la$. Then by Lemma \ref{lb}, $\{w_n\}$ is a bounded in $D^{1,p}(\mathbb R^N)$. So, by Lemma \ref{L1}-$(g_5)$, $\{g(w_n)\}$ is also bounded in $D^{1,p}(\mathbb R^N)$. Therefore, we can assume that \(w_n \rightharpoonup w\) weakly in $D^{1,p}(\mathbb R^N)$, $w_n \ra w$ a.e in $\mathbb R^N$. Since, $g\in C^{\infty}$, then $|g^2(w_n)|^{p}\ra |g^2(w)|^{p}$ a.e in $\mathbb R^N$ and $|g^2(w_n)|^{p}\rightharpoonup |g^2(w)|^{p}$ weakly in $D^{1,p}(\mathbb R^N)$. Hence, we can assume that
	\[|\nabla g^{2}(w_n)|^p \rightharpoonup \omega,\quad |g^2(w_n)|^{p^{*}} \rightarrow \nu,\quad \left(\int_{\mathbb R^N}\frac{|g^2(w(y))|^{p^{*}_{\ba, \mu}}}{|x-y|^{\mu}|x|^{\ba}}\right)\frac{|g(w(x))|^{2p^{*}_{\ba, \mu}}}{|x|^{\ba}}\rightharpoonup \zeta \]
	in the sense of measure. By Lemma \ref{c1}, there exists at most countable set $J$, sequence of points $\{x_j\}_{j\in J} \subset \mathbb R^N$
	and families of positive numbers $\{\nu_j : j\in J\}$, $\{\zeta_j: j\in J\}$ and $\{\omega_j: j\in I\}$ such that
	\begin{align}
		\zeta & = \left(\int_{\mathbb R^N}\frac{|g(w(y))|^{2p^{*}_{\ba, \mu}}}{|x-y|^{\mu}|x|^{\ba}}\right)\frac{|g(w(x))|^{2p^{*}_{\ba, \mu}}}{|x|^{\ba}} +\sum_{j\in J} \zeta_{j} \delta_{x_j},\quad \sum_{j\in J} \zeta_j^{\frac{1}{p^{*}_{\ba, \mu}}} < \infty,\label{l111}\\
		\omega &\geq |\nabla g^2(w)|^p + \sum_{j\in J} \omega_{j} \delta_{x_j}\label{l21}\\
		\nu &\geq  |g(w)|^{2p^{*}} + \sum_{j\in J} \nu_{j} \delta_{x_j}\label{l31}
	\end{align}
	and
	\begin{align}
		S_{\ba, \mu} \zeta_{j}^{\frac{p}{2 p^{*}_{\ba, \mu}}}\leq \omega_j, \;\mbox{and} \; \zeta_{j}^{\frac{N}{2N-2\ba-\mu}} \leq C(N,\ba,\mu)^{\frac{N}{2N-2\ba-\mu}} \nu_j,\label{l41}
	\end{align}
	where \(\delta_x\) is the Dirac-mass of mass \(1\) concentrated at \(x\in\mathbb R^N\).
	
	\noi	Moreover, we can construct a smooth cut-off function \(\psi_{\e,j}\) centered at \(x_j\) such that
	\[0\leq \psi_{\e,j}(x)\leq 1, \psi_{\e,j}(x)=1\;\mbox{in}\; B\left(x_j, \frac{\e}{2}\right),  \psi_{\e,j}(x)=0\;\mbox{in}\; \mathbb R^N\setminus B (x_j,\e), |\nabla \psi_{\e,j}|\leq \frac{4}{\e},\]
	for any \(\e>0\) small.
	
	\noi Let  us set \[v_n:= (1+{2^{p-1} |g(w_n)|^{p}})^{\frac{1}{p}} g(w_n).\] Then \(\{v_n\}\) is bounded in \(D^{1,p}(\mathbb R^N)\). Obviously, \(\langle \mc I_{\la}(w_n), v_n \psi_{\e,j}\rangle \ra 0\) as $n\ra \infty$.
	So, we have
	\begin{align}\label{c0}
		&-\lim_{\e\rightarrow 0} \lim_{n\ra \infty}\left[ a \int_{\mathbb R^N} \frac{g(w_n)}{g^{\prime}(w_n)} |\na w_n|^{p-2}\nabla w_n \nabla \psi_{\e,j} dx + b \left(\int_{\mathbb R^N} \frac{|\nabla w_n|^p}{1+2^{p-1} |g(w_n)|^{p}}\right) \int_{\mathbb R^N}\frac{g(w_n) |\na w_n|^{p-2}\nabla w_n \nabla \psi_{\e,j}}{ (1+{2^{p-1} |g(w_n)|^{p}})^{\frac{1}{p}}} dx\right]\notag\\
		&=\lim_{\e\rightarrow 0} \lim_{n\ra \infty}\left[ a \int_{\mathbb R^N} \left(1+\frac{2^{p-1} |g(w_n)|^{p}}{1+2^{p-1} |g(w_n)|^{p}}\right) |\nabla w_n|^p  \psi_{\e,j} dx + b \left( \int_{\mathbb R^N} \frac{|\nabla w_n|^p}{1+2^{p-1} |g(w_n)|^{p}}\right) \left(\int_{\mathbb R^N}\frac{|\nabla w_n|^p \psi_{\e,j}}{1+{2^{p-1} |g(w_n)|^{p}}} dx\right)\right.\notag\\
		&\quad -\la\left. \int_{\mathbb R^N} f(x)|g(w_n)|^{q} \psi_{\e,j} dx - \int_{\mathbb R^N} \int_{{\mathbb R^N}} \frac{|g(w_{n}(y))|^{2p_{\beta,\mu}^*}|g(w_{n}(x))|^{2p_{\beta,\mu}^*}\psi_{\e,j}(x) }{|y|^\beta|x-y|^\mu|x|^{\ba}}dydx \right].
	\end{align}
	Now the H\"{o}lder inequality and Lemma \ref{L1}-$(g_4)$  yield that
	\begin{align}\label{c3}
		0\leq &\lim_{\e\rightarrow 0} \lim_{n\ra \infty}\left| a \int_{\mathbb R^N} (1+{2^{p-1} |g(w_n)|^{p}})^{\frac{1}{p}} g(w_n) |\nabla w_n|^{p-2}\nabla w_n \nabla \psi_{\e,j} dx\right|\notag\\
		\leq & K \lim_{\e\rightarrow 0} \lim_{n\ra \infty} \int_{\mathbb R^N} |w_n |\nabla w_n|^{p-2}\nabla w_n \nabla \psi_{\e,j}| dx\notag\\
		\leq & K \lim_{\e\rightarrow 0} \lim_{n\ra \infty}\left[\left( \int_{\mathbb R^N} |\na w_n |^p dx\right)^{\frac{p-1}{p}} \left(\int_{\mathbb R^N} |w_n\nabla  \psi_{\e,j}|^{p} dx\right)^{\frac{1}{p}} \right] \notag\\
		\leq & K \lim_{\e\rightarrow 0} \left( \int_{B(x_j, 2\e)} |\na \psi_{\e, j}|^{N} dx\right)^{\frac{p}{N}} \left( \int_{B(x_j, 2\e)} |w|^{\frac{Np}{N-p}} dx\right)^{\frac{N-p}{Np}}\notag\\
		\leq & K \lim_{\e\rightarrow 0} \left( \int_{B(x_j, 2\e)} |w|^{p^*} dx\right)^{\frac{1}{p^*}} =0,
	\end{align}
	Similarly, using the boundedness of \(\{w_n\}\) and the definition of \(\psi_{\e, j}\), we have
	\begin{align}\label{c2}
		\lim_{\e\rightarrow 0} \lim_{n\ra \infty}\left[b \left(\int_{\mathbb R^N} \frac{|\nabla w_n|^p}{1+2^{p-1} |g(w_n)|^{p}} dx\right) \left(\int_{\mathbb R^N}\frac{g(w_n) |\nabla w_n|^{p-2} \nabla w_n \nabla \psi_{\e,j}}{ (1+{2^{p-1} |g(w_n)|^{p}})^{\frac1p}} dx\right)  \right] =0 .
	\end{align}
	One can easily check that, 
	\begin{align}\label{c4}
		\lim_{\e\rightarrow 0} \lim_{n\rightarrow \infty}\int_{\mathbb R^N} f(x) |g(w_n)|^{q} \psi_{\e,j} dx =0.
	\end{align}	 Now by Lemma \ref{L1}-$(g_{11})$, we have
\begin{align}\label{gg}
	|\na g^2(w_n)|^p =|2g(w_n)g'(w_n)\nabla w_n|^p \leq 2 |\na w_n|^p.
\end{align}
	\noi	Plugging the relation together with \eqref{c3}, \eqref{c2} and \eqref{c4} in \eqref{c0},  we deduce
	\begin{align*}
		0 &=\lim_{\e\rightarrow 0} \lim_{n\ra \infty}\left[ a \int_{\mathbb R^N} \left(1+\frac{2^{p-1} |g(w_n)|^{p}}{1+2^{p-1} |g(w_n)|^{p}}\right) |\nabla w_n|^p  \psi_{\e,j} dx + b \left( \int_{\mathbb R^N} \frac{|\nabla w_n|^p}{1+2^{p-1} |g(w_n)|^{p}}\right) \left(\int_{\mathbb R^N}\frac{|\nabla w_n|^p \psi_{\e,j}}{1+{2^{p-1} |g(w_n)|^{p}}} dx\right)\right.\notag\\
		&\quad -\la\left. \int_{\mathbb R^N} f(x)|g(w_n)|^{q} \psi_{\e,j} dx - \int_{\mathbb R^N} \int_{{\mathbb R^N}} \frac{|g(w_{n}(y))|^{2p_{\beta,\mu}^*}|g(w_{n}(x))|^{2p_{\beta,\mu}^*}\psi_{\e,j}(x) }{|y|^\beta|x-y|^\mu|x|^{\ba}}dydx \right]\\
		&\geq \lim_{\e\ra 0}\lim_{n\ra\infty} \left\{a\int_{\mathbb R^N} |\na g^2(w_n)|^{p} \psi_{\e,j} dx - \int_{\mathbb R^N} \int_{\mathbb R^N} \frac{|g(w_{n}(y))|^{2p_{\beta,\mu}^*}|g(w_{n}(x))|^{2p_{\beta,\mu}^*}\psi_{\e,j}}{|y|^\beta|x-y|^\mu|x|^{\ba}} dxdy\right\}\\
		& \geq \lim_{\e\ra 0} \lim_{n\ra 0} \left\{a \int_{\mathbb R^N} \psi_{\e,j} d\omega - \int_{\mathbb R^N} \psi_{\e,j} d\zeta \right\}\\
		&\geq a \omega_j - \zeta_j.
	\end{align*}
	Combining this with \eqref{l41}, it follows that 
	\[ \mbox{either}\quad \omega_{j}\geq \left( a  S_{\ba,\mu}^{\frac{2N-2\ba-\mu}{N-p}}\right)^{\frac{N-p}{N-2\ba-\mu +p}}\quad \mbox{or}\quad \omega_j=0. \]
	
	Now we claim that the first case can not occur. Suppose not, then there exists $j_0\in J$ such that $\omega_{j_0}\geq \left( aS_{\ba,\mu}^{\frac{2N-2\ba-\mu}{N-p}}\right)^{\frac{N-p}{N-2\ba-\mu +p}}$.
	Now the H\"{o}lder inequality, \eqref{se} and the Young inequality yield that
	\begin{align}\label{e1}
		\la \int_{\mathbb R^N} f(x) |g(w)|^{q} dx &\leq \la \|f\|_{\frac{2p^*}{2p^*-q}} S^{-\frac{q}{2p}} \|g^2(w)\|^{\frac{q}{2}}= \left(\left[\left(\frac{1}{p}- \frac{1}{2p_{\ba,\mu}^*}\right)\frac{a}{2}\left(\frac{1}{q}- \frac{1}{4p_{\ba,\mu}^*}\right)^{-1}\right]^{\frac{q}{2p}}\|g^2(w)\|^{\frac{q}{2}}\right)\notag\\
		& \quad\quad \left(\left[\left(\frac{1}{p}- \frac{1}{2p_{\ba,\mu}^*}\right) \frac{a}{2}\left(\frac{1}{q}- \frac{1}{4p_{\ba,\mu}^*}\right)^{-1}\right]^{\frac{-q}{2p}} \la \|f\|_{\frac{2p^*}{2p^*-q}} S^{-\frac{q}{2p}}\right)\notag\\
		&\leq \left(\frac{1}{p}- \frac{1}{2p_{\ba,\mu}^*}\right) \frac{a}{2}\left(\frac{1}{q}- \frac{1}{4p_{\ba,\mu}^*}\right)^{-1} \|g^2(w)\|^p \notag\\
		& \quad \quad + \frac{2p-q}{2p}\left[\left(\frac{1}{q}- \frac{1}{4p_{\ba,\mu}^*}\right) \frac{2}{aS}\left(\frac{1}{p}- \frac{1}{2p_{\ba,\mu}^*}\right)^{-1}\right]^{\frac{q}{2p-q}} {\la}^{\frac{2p}{2p-q}} \|f\|_{\frac{2p^*}{2p^*-q}}^{\frac{2p}{2p-q}}.
	\end{align}
	Using \eqref{e1}, we have
	{\small \begin{align}\label{e6}
			0&>c= \lim_{n\ra \infty}\left(\mc I_{\la}(w_n) - \frac{1}{4 p_{\ba,\mu}^*}\left \langle \mc I^{\prime}_{\la}(w_n), (1+2^{p-1}|g(w_n)|^p)^{\frac{1}{p}} g(w_n) \right\rangle \right)\notag\\
			&=	\lim_{n\ra \infty}\left\{a \int_{\mathbb R^N} \left[\frac1p -\frac{1}{4p_{\ba,\mu}^{*}} \left(1+\frac{2^{p-1} |g(w_n)|^{p}}{1+2^{p-1} |g(v_n)|^{p}}\right)\right] |\nabla w_n|^p  + \left( \frac{1}{2p}-\frac{1}{4p_{\ba,\mu}^{*}}\right) b \left(\int_{\mathbb R^N} |g^{\prime}(w_n)|^p |\nabla w_n|^p \right)^2\right.\notag \\
			&\quad \quad +\left. \left(\frac{1}{4p_{\ba,\mu}^{*}} -\frac{1}{q}\right) \la\int_{\mathbb R^N} f(x)|g(w_n)|^{q} dx\right\} \notag\\
			&\geq \; \lim_{n\ra \infty}\left\{ \frac{a}{2} \left(\frac1p -\frac{1}{2p_{\ba,\mu}^{*}}\right) \int_{\mathbb R^N} |\nabla g^2(w_n)|^p dx  - \left(\frac{1}{q}- \frac{1}{4p_{\ba,\mu}^{*}} \right) \la \int_{\mathbb R^N} f(x)|g(w_n)|^{q} dx \right\}\notag \\
			& \geq \left(\frac{1}{p}- \frac{1}{2p_{\ba,\mu}^*}\right) \frac{a}{2} \left(\|g^2(w)\|^{p}+ \sum_{j\in J} \omega_j \right)- \left(\frac{1}{q}- \frac{1}{4p_{\ba,\mu}^*}\right) \la \int_{\mathbb R^N} f(x) |g(w_n)|^{q} dx\notag \\
			&\geq \left(\frac{1}{p}- \frac{1}{2p_{\ba,\mu}^*}\right) \frac{a}{2} \omega_{j_0} - \frac{2p-q}{2p}\left[\left(\frac{1}{q}- \frac{1}{4p_{\ba,\mu}^*}\right) \frac{2}{aS}\left(\frac{1}{p}- \frac{1}{2p_{\ba,\mu}^*}\right)^{-1}\right]^{\frac{2p}{2p-q}} {\la}^{\frac{2p}{2p-q}} \|f\|_{\frac{2p^*}{2p^*-q}}^{\frac{2p}{2p-q}}\notag\\
			&\geq \left(\frac{1}{2p}- \frac{1}{4p_{\ba,\mu}^*}\right)\left( a  S_{\ba,\mu}\right)^{\frac{p_{\ba,\mu}^*}{p_{\ba,\mu}^*-1}} - \frac{2p-q}{2p}\left[\left(\frac{1}{q}- \frac{1}{4p_{\ba,\mu}^*}\right) \frac{2}{aS}\left(\frac{1}{p}- \frac{1}{2p_{\ba,\mu}^*}\right)^{-1}\right]^{\frac{q}{2p-q}} {\la}^{\frac{2p}{2p-q}} \|f\|_{\frac{2p^*}{2p^*-q}}^{\frac{2p}{2p-q}}.
	\end{align}}
	Choose $\la_1>0$ so small such that for every $\la\in(0,\la_1)$, the right hand side of \eqref{e6} is greater than zero, which gives a contradiction.

	To obtain the possible concentration of mass at infinity, similarly, we can define a cut-off function \(\psi_{R}\in C^{\infty}(\mathbb R^N)\) such that \(\psi_{R}(x)=0\) on \(|x|<R\), \(\psi_{R}(x)=1\) on \(|x|>R+1\) and \(|\nabla \psi_{R}|\leq \frac{2}{R}\). Let 
	\begin{align*}
		\omega_{\infty} := \lim_{R\ra \infty} \limsup_{n\ra \infty} \int_{|x|\geq R} |\na g^2(w_n)|^p dx, \quad \nu_{\infty}:= \lim_{R\ra \infty} \limsup_{k\ra \infty} \int_{|x|\geq R} |g(w_n)|^{2p^*} dx\\
		\zeta_{\infty} := \lim_{R\ra \infty} \limsup_{n\ra \infty} \int_{|x|> R}\left(\int_{\mathbb R^N}\frac{|g(w_{n})(y)|^{2p^{*}_{\ba,\mu}}}{|y|^{\ba}|x-y|^{\mu}}dy\right)\frac{|g(w_{n})(x)|^{2p^{*}_{\ba,\mu}}}{|x|^{\ba}} dx.
	\end{align*}
	
	Now applying Proposition \ref{P1}, the H\"{o}lder inequality and Lemma \ref{L1}-$(g_6)$, we deduce
	\begin{align*}
		\zeta_{\infty}&= \lim_{R\ra \infty} \lim_{n\ra \infty} \left(\int_{\mathbb R^N}\frac{|g(w_{n})(y)|^{2p^{*}_{\ba,\mu}}}{|y|^{\ba}|x-y|^{\mu}}\right)\frac{|g(w_{n})(x)|^{2p^{*}_{\ba,\mu}}}{|x|^{\ba}} \psi_{R}(x) dx\\
		&\leq C({N,\ba,\mu} )\lim_{R\ra \infty} \lim_{n\ra \infty} |g^2(w_n)|^{p^{*}_{\ba,\mu}}_{p^*} \left(\int_{\mathbb R^N} |g(w_n(x))|^{2p*}\psi_{R}(x) dx\right)^{\frac{p^{*}_{\ba,\mu}}{p^{*}}}
		\leq K \nu_{\infty}^{\frac{p_{\ba,\mu}^{*}}{p^*}}.
	\end{align*}
	Using the fact \(\langle \mc I^{\prime}_{\la}(w_n), \frac{g(w_n)}{g^{\prime}(w_n)} \psi_{R}\rangle \ra 0\), we get
	\begin{align}\label{3.20}
		&- \lim_{n\ra \infty}\left[ a \int_{\mathbb R^N} \frac{g(w_n)}{g^{\prime}(w_n)} |\nabla w_n|^{p-2}\nabla w_n \nabla \psi_{R} dx + b \left(\int_{\mathbb R^N} |g^{\prime}(w_n)|^{p} |\nabla w_n|^p dx\right) \left(\int_{\mathbb R^N}\frac{g(w_n)|\nabla w_n|^{p-2} \nabla w_n \nabla \psi_{R}}{ (1+{2^{p-1} |g(w_n)|^{p}})^{\frac{1}{p}}} dx\right) dx \right]\notag\\
		&=\lim_{n\ra \infty}\left[ a \int_{\mathbb R^N} \left(1+\frac{2^{p-1} |g(w_n)|^{p}}{1+2^{p-1} |g(w_n)|^{p}}\right) |\nabla w_n|^p  \psi_{R} dx + b \left( \int_{\mathbb R^N} \frac{|\nabla w_n|^p}{1+2^{p-1} |g(w_n)|^{p}} dx\right) \left(\int_{\mathbb R^N}\frac{|\nabla w_n|^{p-2}\nabla w_n \nabla \psi_{R}}{1+{2^{p-1} |g(w_n)|^{p}}} dx\right)\right.\notag\\
		&\quad \quad-\la\left. \int_{\mathbb R^N} f(x)|g(w_n)|^{q}\psi_{R} dx - \int_{\mathbb R^N} \int_{\mathbb R^N} \frac{|g(w_{n}(y))|^{2p_{\beta,\mu}^*}|g(w_{n}(x))|^{2p_{\beta,\mu}^*}\psi_{R}(x) }{|y|^\beta|x-y|^\mu|x|^{\ba}}dydx  \right].
	\end{align}
	One can easily show that
	\[\lim_{R\ra \infty} \lim_{n\ra \infty} a \int_{\mathbb R^N} (1+2^{p-1} |g(w_n)|^{p})^{\frac{1}{p}} g(w_n)|\nabla w_n|^{p-2} \nabla w_n \nabla \psi_R dx =0,\]
	\[\lim_{R\ra \infty}\lim_{n\ra \infty} \left[b \left(\int_{\mathbb R^N}\frac{|\na w_n|^p}{1+2^{p-1} |g(w_n)|^{p}} dx\right) \left(\int_{\mb R^N} \frac{g(w_n)
		|\nabla w_n|^{p-2}\nabla w_n \nabla \psi_{R}}{(1+2^{p-1} |g(w_n)|^{p})^{\frac{1}{p}}}\right) \right]=0,\]
	and
	\[\lim_{R\ra \infty}\lim_{n\ra \infty} \int_{\mathbb R^N}f(x)|g(w_n(x))|^{q} \psi_{R}(x) dx =0.\]
	Using the above in \eqref{3.20}, we obtain
	\begin{align}
		0=&\lim_{R\rightarrow \infty} \lim_{n\ra \infty}\left[ a \int_{\mathbb R^N} \left(1+\frac{2^{p-1} |g(w_n)|^{p}}{1+2^{p-1} |g(w_n)|^{p}}\right) |\nabla w_n|^p  \psi_{R} dx- \int_{\mb R^N} \frac{|g(w_{n}(y))|^{2p_{\beta,\mu}^*}|g(w_{n}(x))|^{2p_{\beta,\mu}^*}\psi_{R}(x) }{|y|^\beta|x-y|^\mu|x|^{\ba}}dydx\right.\notag\\
		&\quad\left. + b \left(\int_{\mathbb R^N} \frac{|\nabla w_n|^p}{1+2^{p-1} |g(w_n)|^{p}} dx\right) \left(\int_{\mathbb R^N}\frac{|\nabla w_n|^{p}  \psi_{R}}{ \sqrt{1+{2^{p-1} |g(w_n)|^{p}}}} dx\right) dx \right] \notag\\
		\geq& \lim_{R\ra \infty}\lim_{n\ra \infty}\left[ a \int_{\mathbb R^N} |\nabla g^{2}(w_n)|^p  \psi_{R} dx - \int_{\mathbb R^N} \frac{|g(w_{n}(y))|^{2p_{\beta,\mu}^*}|g(w_{n}(x))|^{2p_{\beta,\mu}^*}\psi_{R}(x) }{|y|^\beta|x-y|^\mu|x|^{\ba}}dydx  \right]\notag\\
		=& a \omega_{\infty}- K \nu_{\infty}^{\frac{p_{\beta,\mu}^{*}}{p^*}}.
	\end{align}
	Thus, $a \omega_{\infty} \leq K \nu_{\infty}^{\frac{p_{\ba,\mu}^{*}}{p^*}}$. This together with Lemma \ref{c1} yields that
	\begin{align}
		\omega_{\infty}\geq \left( K^{-1} a S^{\frac{p_{\beta, \mu}^{*}}{p}} \right)^{\frac{p}{p_{\ba,\mu}^{*}-p}}\, \mbox{or}\; \omega_{\infty}=0.
	\end{align}
	If $\omega_{\infty}\geq \left( K^{-1} a S^{\frac{p_{\beta,\mu}^{*}}{p}} \right)^{\frac{p}{p_{\ba,\mu}^{*}-p}},$ then we have
	\begin{align}\label{e7}
		0&> c = \lim_{R\ra \infty}\lim_{n\ra \infty} \left(\mc I_{\la} (w_n) -\frac{1}{4p_{\ba,\mu}^*} \left\langle \mc I^{\prime}_{\la}(w_n), \frac{g(w_n)}{g^{\prime}(w_n)}\right\rangle\right)\notag\\
		&\geq \lim_{R\ra \infty}\lim_{n\ra \infty} \left\{ \left(\frac{1}{p}- \frac{1}{2p_{\ba,\mu}^*}\right)a \int_{\mathbb R^N} |\nabla w_n|^p \psi_{R} dx -  \left(\frac{1}{q}- \frac{1}{4 p_{\ba, \mu}^{*}}\right)\int_{\mathbb R^N} f(x)|g(w_n)|^{q}  dx\right\}\notag\\
		&\geq \lim_{\e\ra 0}\lim_{n\ra \infty} \left\{ \left(\frac{1}{p}- \frac{1}{2 p_{\ba,\mu}^*}\right)\frac{a}{2} \int_{\mathbb R^N} |\nabla g^{2}(w_n)|^p dx -  \left(\frac{1}{q}- \frac{1}{4 p_{\ba, \mu}^{*}}\right)\int_{\mathbb R^N} f(x)|g(w_n)|^{q} dx\right\}\notag\\
		&\geq \left(\frac{1}{2p}- \frac{1}{4p_{\ba,\mu}^*}\right)\left( a  S\right)^{\frac{p_{\ba,\mu}^*}{p_{\ba,\mu}^*-p}}K^{\frac{-p}{p_{\ba,\mu}^*-p}} - \frac{2p-q}{2p}\left[\left(\frac{1}{q}- \frac{1}{4p_{\ba,\mu}^*}\right) \frac{2}{aS}\left(\frac{1}{p}- \frac{1}{2p_{\ba,\mu}^*}\right)^{-1}\right]^{\frac{q}{2p-q}} {\la}^{\frac{2p}{2p-q}} \|f\|_{\frac{2p^*}{2p^*-q}}^{\frac{2p}{2p-q}}.
	\end{align}
	Choose $\la_2>0$ so small such that for every $\la\in(0,\la_2)$, the right hand side of \eqref{e7} is greater than zero, which gives a contradiction.
		Now from the above arguments, for any $c<0$, there exist $\la^*=\min\{\la_1,\la_2\}>0$, we have \(\omega_j=0\) for all $j\in J$ and $\omega_{\infty}=0$ for all \(\la\in (0, \la^*)\). Hence
	\begin{align*}
		\lim_{k\ra \infty}\int_{\mathbb R^N}\int_{\mathbb R^N} \frac{|g(w_n(x))|^{2p^*_{\ba,\mu}}|g(w_n(y))|^{2p^*_{\ba,\mu}}}{|x|^{\ba}|x-y|^{\mu}|y|^{\ba}} dx dy = \int_{\mathbb R^N}\int_{\mathbb R^N} \frac{|g(w(x))|^{2p^*_{\ba,\mu}}|g(w(y))|^{2p^*_{\ba,\mu}}}{|x|^{\ba}|x-y|^|y|^{\ba}} dx dy .
	\end{align*}
	and
	\begin{align*}
		\lim_{k\ra \infty}\int_{\mathbb R^N} f(x) (|g(w_n(x))|^{q} - |g(w(x))|^{q}) dx \leq \|f\|_{\frac{2p^*}{2p^*-q}} \||g(w_n(x))|^{q} - |g(w(x))|^{q}\|_{\frac{2p^*}{q}} = 0.
	\end{align*}
	Since $\{w_n\}$ is bounded in $D^{1,p}(\mb R^N)$ and $ \mc I^{\prime}_{\la}(w)=0$, the weak lower semicontinuity of the norm, Lemma \ref{sk} and the Br{e}zis-Lieb Lemma (see \cite{bz}) yield that as $n\ra \infty$,
	\begin{align*}
		o_n(1)\|w_n\|&= \left \langle \mc I^{\prime}_{\la}(w_n), (1+2^{p-1}|g(w_n)|^p)^{\frac{1}{p}} g(w_n) \right\rangle \\
		&= a \int_{\mathbb R^N} \left(1+\frac{2^{p-1} |g(w_n)|^{p}}{1+2^{p-1} |g(w_n)|^{p}}\right) |\nabla w_n|^p  dx + b \left( \int_{\mathbb R^N} \frac{|\nabla w_n|^p}{1+2^{p-1} |g(w_n)|^{p}} dx\right)^2\\
		& \quad - \la \int_{\mathbb R^N} f(x) |g(w_n)|^{q} dx - \int_{\mathbb R^N}\int_{\mathbb R^N} \frac{|g(w_n(x))|^{2p^*_{\ba,\mu}}|g(w_n(y))|^{2p^*_{\ba,\mu}}}{|x|^{\ba}|x-y|^{\mu}|y|^{\ba}} dx dy  \\
		&= a\|w_n\|^{p}+ a \int_{\mathbb R^N} \frac{2^{p-1} |g(w_n)|^{p}}{1+2^{p-1} |g(w_n)|^{p}} |\nabla w_n|^p  dx + b \left( \int_{\mathbb R^N} \frac{|\nabla w_n|^p}{1+2^{p-1} |g(w_n)|^{p}} dx\right)^2\\
		&\quad - \la \int_{\mathbb R^N} f(x) |g(w_n)|^{q} dx - \int_{\mathbb R^N}\int_{\mathbb R^N} \frac{|g(w_n(x))|^{2p^*_{\ba,\mu}}|g(w_n(y))|^{2p^*_{\ba,\mu}}}{|x|^{\al}|x-y|^{\mu}|y|^{\ba}} dx dy  \\
		&\geq a(\|w_n - w\|^p) + a \|w\|^p + b  \left( \int_{\mathbb R^N} \frac{|\nabla w|^p}{1+2^{p-1} |g(w)|^{p}} dx\right)^2- \la \int_{\mathbb R^N} f(x) |g(w)|^{q} dx \\
		&\quad - \int_{\mathbb R^N}\int_{\mathbb R^N} \frac{|g(w(x))|^{2p^*_{\ba,\mu}}|g(w(y))|^{2p^*_{\ba,\mu}}}{|x|^{\ba}|x-y|^{\mu}|y|^{\ba}} dx dy \\
		&= a \|w_n -w\|^p + o_n(1)\|w\|.
	\end{align*}
	Thus $\{w_n\}$ converges strongly to $w$ in $D^{1,p}(\mathbb R^N)$. This completes the proof of the Lemma.
\end{proof}
\begin{lemma}\label{ce}
  Assume that $q=2p$ and \eqref{assumption1} hold. Let  \(\{w_n\}\) be a $(PS)_c$ sequence for $\mc I_{\la}$ in $D^{1,p}(\mathbb R^N)$ with $$c<c^*:=\frac{1}{4p}\left( aS_{\beta,\mu}\right)^{\frac{p_{\ba,\mu}^*}{p_{\ba,\mu}^*-1}}.$$ Then for all \(\la\in(0,aS\|f\|_{\frac{p^*}{p^*-p}}^{-1})\), $\{w_n\}$ satisfies the $(PS)_c$  condition.
\end{lemma}
\begin{proof}
	For each \(w\in D^{1,p}(\mathbb R^N)\),	using Lemma \ref{L1}-$(g_6)$, the H\"{o}lder inequality and Sobolev inequality \eqref{se}, we obtain
	\begin{align*}
		\int_{\mb R^N} f(x)|g(w)|^{2p} dx \leq S^{-1} \|f\|_{\frac{p^*}{p^*-p}}  \|g^2(w)\|^p.
	\end{align*}
	Let $\{w_n\}$ be a $(PS)_c$	for $\mc I_\la$ for $c<c^*$. Then $\{w_n\}$ is bounded from Lemma \ref{lbb}. Now using the last estimate, for  all \(\la\in(0,aS\|f\|_{\frac{p^*}{p^*-p}}^{-1})\), arguing similarly as in Lemma \ref{lb}, in substitute of \eqref{e6}, we obtain
	\begin{align*}
		c^*> c &= \lim_{n\ra \infty} \left(\mc I_{\la}(w_n) - \frac{1}{4p}\left\langle \mc I^{\prime}_{\la}(w_n), \frac{g(w_n)}{g^{\prime}(w_n)}\right\rangle\right)\\
				&=	\lim_{n\ra \infty}\left\{a \int_{\mathbb R^N} \left[\frac1p -\frac{1}{4p} \left(1+\frac{2^{p-1} |g(w_n)|^{p}}{1+2^{p-1} |g(v_n)|^{p}}\right)\right] |\nabla w_n|^p  + \left( \frac{1}{2p}-\frac{1}{4p}\right) b \left(\int_{\mathbb R^N} |g^{\prime}(w_n)|^p |\nabla w_n|^p \right)^2\right.\notag \\
			&\quad \quad +\left. \left(\frac{1}{4p} -\frac{1}{2p}\right) \la\int_{\mathbb R^N} f(x)|g(w_n)|^{q} dx+\left(\frac{1}{4p} -\frac{1}{4p_{\ba,\mu}^*}\right)\int_{\mathbb R^N}\int_{\mathbb R^N} \frac{|g(w(x))|^{2p^*_{\ba,\mu}}|g(w(y))|^{2p^*_{\ba,\mu}}}{|x|^{\ba}|x-y|^{\mu}|y|^{\ba}} dx dy\right\} \notag\\
		& \geq 	\lim_{n\ra \infty}\left\{ \frac{a}{2p} \|w_n\|^p - \left(\frac{1}{2p}- \frac{1}{4p}\right)  \la S^{-1} \|f\|_{\frac{p^*}{p^*-p}} \|g^2(w_n)\|^p \right\}\\
			& \geq 	\lim_{n\ra \infty}\left\{ \frac{a}{4p} \|g^2(w_n)\|^p\right\} - \left(\frac{1}{2p}- \frac{1}{4p}\right)\la S^{-1} \|f\|_{\frac{p^*}{p^*-p}} \|g^2(w)\|^p \\
				& \geq \frac{a}{4p} w_{i_0} +  \frac{1}{4p}(a-\la S^{-1} \|f\|_{\frac{p^*}{p^*-p}}) \|g^2(w)\|^p \\
		&\geq \frac{1}{4p} a w_{i_0} \geq \frac {1}{4p}\left( aS_{\ba,\mu}\right)^{\frac{p_{\ba,\mu}^*}{p_{\ba,\mu}^*-1}}:= c^*,
	\end{align*}
	which is absurd. Now the rest of the proof follows in similar manner as in the proof of Lemma \ref{cc}.
\end{proof}
\begin{lemma}\label{cf}
Let  $2p<q<2p^*$ and \eqref{assumption1} hold. Suppose that $\{w_n\}$ is a $(PS)_c$ sequence for $\mc I_{\la}$ in $D^{1,p}(\mathbb R^N)$ with $$c<c^{**}:=\left(\frac{1}{2p}-\frac{1}{q}\right)\left( aS_{\ba,\mu}\right)^{\frac{p_{\ba,\mu}^*}{p_{\ba,\mu}^*-1}}.$$ Then $\{w_n\}$ satisfies $(PS)_c$ condition.
\end{lemma}

\begin{proof}
	Let $\{w_n\}$ a $(PS)_c$	for $\mc I_\la$ for $c<c^{**}$.  Then by Lemma \ref{lbbb}, we have $\{w_n\}$ is bounded. Now  following the similar arguments as in Lemma \ref{lb}, in place of \eqref{e6}, we get
	\begin{align*}
		c^{**}> c &= \lim_{n\ra \infty} \left(I_{\la}(w_n) - \frac{1}{q}\left\langle \mc I^{\prime}_{\la}(w_n), \frac{g(w_n)}{g^{\prime}(w_n)}\right\rangle\right)\\
		& \geq \left(\frac{1}{p}-\frac{2}{q}\right)\frac{a}{2} w_{i_0}
		\geq \left(\frac{1}{2p}-\frac{1}{q}\right)\left( aS_{\ba,\mu}\right)^{\frac{p_{\ba,\mu}^*}{p_{\ba,\mu}^*-1}}:= c^{**},
	\end{align*}
	which is a contradiction. The rest of the proof follows as in the proof of Lemma \ref{cc}.
\end{proof}

	\section{\bf Proof of Theorem \ref{main.result.1}}
\noi In this section, we give proof of Theorem \ref{main.result.1}. Before proving our result, first we recall the definition of genus.
\begin{definition} Let $X$ be a Banach space and $A$ be a subset of $X$. The set $A$ is said to be symmetric if $u \in A$ implies $-u \in A.$ For a closed symmetric set $A$ which does not contain the
	origin, we define a genus $\gamma (A)$ of $A$ by the smallest integer $k$ such that there exists an
	odd continuous mapping from $A$ to $R^k \setminus \{0\}$. If there does not exist such  $k$, we define
	$\gamma (A)=\infty$ . Moreover, we set $\gamma (\emptyset)=0$ .
	\end{definition}
For any $k\in\mathbb N$, let us define the set $\Sigma_k$ as 
\[\Sigma_k:=\{A\;:\; A\subset X \text{ is closed symmetric }, 0\not\in A,\; \gamma (A)\geq k \}.\] 	Now to prove Theorem \ref{main.result.1}, we use a result by Kajikiya ( see \cite[ Theorem 1]{kajikiya}), which is an extension  of the symmetric mountain pass theorem.
\begin{theorem}\label{sym mt}
Let $X$ be an infinite dimensional Banach space and $\mc J\in C^1(X,\mb R)$. Suppose that the following hypotheses hold.
\begin{itemize}
\item[$(\mathcal A_1)$] The functional $\mc J$ is even and bounded from below in $X$, $\mc J(0)=0$ and $\mc J$ satisfies the local Palais-Smale condition.
\item [$(\mc A_2)$] For each $k\in \mathbb N,$ there exists $A_k\in \Sigma_k$ such that $$\sup_{u\in A_k} \mc J(u)<0.$$
\end{itemize} Then $\mc J$ admits a sequence of critical points $\{u_k\}$ in $X$ such that $u_k\not =0$, $\mc J(u_k)\leq 0$ for each $k$  and $u_k \to 0$ in $X$ as $k\to\infty$.
\end{theorem}
\begin{proposition}\label{infty}
Let  \eqref{assumption1} hold. If $w\in D^{1,p}(\mb R^N)$ is a nontrivial weak solution to \eqref{2p}, then $w\in L^\infty(\mb R^N)$. Moreover, if we consider $f\in L^{\infty}(\mb R^N)$ and $2p<q<2p^*$, then any nontrivial weak solution $w\in D^{1,p}(\mb R^N)$ to \eqref{2p} belongs to $L^{\infty}(\mb R^N)\cap C^{1, r} ( B_R(0))$, for all $R>0$ and for some $r:=r(R)\in (0,1)$.

\end{proposition}
\begin{proof}
Let $w\in  D^{1,p}(\mb R^N)$	be a nontrivial weak solution to \eqref{2p}. Without loss of generality let us assume $w\geq 0$. For any  real number $M>0$, we define the function $$v_M:=\min\{w(x),\,M\}.$$ We consider the test function $v=v_M^{kp+1},\; k\geq0$. Clearly $v_M\in D^{1,p}(\mb R^N)\cap L^\infty(\mb R^N)$. Now using $v$ as test function in the weak formulation \eqref{2.5} and using Lemma \ref{L1}-$(g_3)$, we get
\begin{align}\label{r1}
&a\int_{{\mathbb R^N}} |\nabla w|^{p-2}\nabla w\nabla v_M^k dx+ b\int_{\mathbb R^N}|g^{\prime}(w)|^p|\nabla w|^p dx \int_{\mb R^N}|g^{\prime}(w)|^p|\nabla w|^{p-2}\nabla w\nabla v_M^{kp+1 } dx \notag\\
&= -b\int_{\mathbb R^N}|g^{\prime}(w)|^p|\nabla w|^p dx \int_{\mb R^N} |g^{\prime}(w)|^{p-2} g^{\prime}(w) g^{\prime\prime}(w)|\nabla w|^p v_M^{kp+1} dx\notag\\
&\qquad+\la\int_{\mb R^N}  f(x) |g(w)|^{q-2} g(w) g^{\prime}(w) v_M dx+\int_{{\mathbb R^N}} \left(\ds\int_{\mathbb R^N}\frac{|g(w)|^{2p_{\beta,\mu}^{*}}}{|x-y|
	^{\mu}|y|^{\beta}}dy\right)\frac{|g(w)|^{2p_{\beta,\mu}^{*}-2} g(w)}{|x|^{\beta}} g^{\prime}(w) v_M^{kp+1} dx\notag\\
&\leq b\| w\|^p  \int_{\mb R^N} |g^{\prime}(w)|^{p-1}  g^{\prime\prime}(w)|\nabla w|^p v_M^{kp+1}(x) dx\notag\\
&\qquad+\la\int_{\mb R^N}  f(x) |g(w)|^{q-1}  g^{\prime}(w) v_M^{kp+1}(x) dx
+\int_{{\mathbb R^N}} \left(\ds\int_{\mathbb R^N}\frac{|g(w)|^{2p_{\beta,\mu}^{*}}}{|x-y|
	^{\mu}|y|^{\beta}}dy\right)\frac{|g(w)|^{2p_{\beta,\mu}^{*}-1} }{|x|^{\beta}} g^{\prime}(w) v_M^{kp+1}(x) dx.
\end{align}
Now we estimate the integral expressions in the left hand side of \eqref{r1}:\\
 Using \eqref{se}, we get 
\begin{align}\label{r2}
a\int_{{\mathbb R^N}} |\nabla w|^{p-2}\nabla w\nabla v_M^{kp+1} dx
&=a(kp+1)\int_{{\mathbb R^N}} |\nabla v_M|^{p}v_M^{kp }dx\notag\\
&=a\frac{(kp+1)}{(k+1)^p}\int_{{\mathbb R^N}} |\nabla v_M^{k+1}|^{p}dx\notag\\
&\geq a\frac{(kp+1)}{(k+1)^p}S^{1/p}\|v_M\|^{(k+1)p}_{(k+1)p^*}.
\end{align}
Similarly, 
\begin{align}\label{r3}
	\int_{\mathbb R^N}|g^{\prime}(w)|^p|\nabla w|^p dx \int_{\mb R^N}|g^{\prime}(w)|^p|\nabla w|^{p-2}\nabla w\nabla v_M^{kp+1 } dx&=b\|g(w)\|^p\int_{\mb R^N}|g^{\prime}(w)|^p|\nabla w|^{p-2}\nabla w\nabla v_M^{kp+1 } dx\notag\\
	&=b\frac{(kp+1)}{(k+1)^p}\|g(w)\|^p\int_{\mb R^N}|g^{\prime}(w)|^p |\nabla v_M^{k+1}|^{p}dx\geq 0.
\end{align}
Next, we estimate the integral expressions in the right hand side of \eqref{r1}:\\ Recalling Lemma \ref{L1}-$(g_3), (g_{11})$ and $(g_{12})$, we deduce
\begin{align}\label{r4}
	&b\| w\|^p  \int_{\mb R^N} |g^{\prime}(w)|^{p-1}  g^{\prime\prime}(w)|\nabla w|^p v_M^{kp+1}(x) dx\notag\\
	&\leq b\| w\|^p  \int_{\mb R^N}  |g^{\prime}(w)|^{p+2}|g^{\prime}(w)|^{p-1}  |g(w)|^{p-1}|\nabla w|^p M^{kp+1} dx\notag\\
	&\leq \frac {b }{2^{\frac{(p-1)^2}{p}}}\| w\|^p  \int_{\mb R^N} |\nabla w|^p M^{kp+1} dx\notag\\
	&=\frac {b }{2^{\frac{(p-1)^2}{p}}}M^{kp+1}\|w\|^{2p}.
\end{align}
Applying Lemma \ref{L1}-$(g_4), (g_6)$ and \eqref{se}, we obtain
\begin{align}\label{r5}
	\la\int_{\mb R^N}  f(x) |g(w)|^{q-1}  g^{\prime}(w) v_M^{kp+1}(x) dx
	&\leq \la\int_{\mb R^N}  f(x) |g(w)|^{q} v_M^{kp} dx\notag\\
	&\leq \la 2^{q/{2p}}\int_{\mb R^N} f(x)|w|^{q/2} M^{kp} dx\notag\\
	&\leq \la 2^{q/{2p}} M^{kp} \|f\|_{\frac{2p^*}{2p^*-q}}\|w\|_{p^*}^{q/2}\notag\\
	&\leq \la 2^{q/{2p}} M^{kp} \|f\|_{\frac{2p^*}{2p^*-q}}S^{-q/{2p}}\|w\|^{q/2}.
\end{align}
Again employing Lemma \ref{L1}-$(g_4), (g_6)$, \eqref{se} and recalling Proposition \ref{P1} and applying the H\"older inequality, we deduce
\begin{align}\label{r6}
	&\int_{{\mathbb R^N}} \left(\ds\int_{\mathbb R^N}\frac{|g(w)|^{2p_{\beta,\mu}^{*}}}{|x-y|
		^{\mu}|y|^{\beta}}dy\right)\frac{|g(w)|^{2p_{\beta,\mu}^{*}-1} }{|x|^{\beta}} g^{\prime}(w) v_M^{kp+1}(x) dx\notag\\
	&\leq \int_{{\mathbb R^N}} \left(\ds\int_{\mathbb R^N}\frac{|g(w)|^{2p_{\beta,\mu}^{*}}}{|x-y|
		^{\mu}|y|^{\beta}}dy\right)\frac{|g(w)|^{2p_{\beta,\mu}^{*}} }{|x|^{\beta}} w^{kp}(x) dx\notag\\
		&\leq 2^{\frac{p_{\beta,\mu}^*}{p}} C(N,p,\mu,\beta)	\|w\|^{p_{\beta,\mu}^*}_{p^*}\left(\int_{\mb R^N} |w|^{(p_{\beta,\mu}^*+kp)\frac{p^*}{p_{\beta,\mu}^*}}\right)^{\frac{p_{\beta,\mu}^*}{p^*}}\notag\\
		&:=C\left(\int_{\{w\leq \tau\}} |w|^{(p_{\beta,\mu}^*+kp)\frac{p^*}{p_{\beta,\mu}^*}}
		+\int_{\{w> \tau\}} |w|^{(p_{\beta,\mu}^*+kp)\frac{p^*}{p_{\beta,\mu}^*}}\right)^{\frac{p_{\beta,\mu}^*}{p^*}}\notag\\
		&\leq C\left[\tau^{p_{\beta,\mu}^*-p}\left(\int_{\{w\leq \tau\}} (|w|^{(k+1)p})^{\frac{p^*}{p_{\beta,\mu}^*}}\right)^{\frac{p_{\beta,\mu}^*}{p^*}}
		+\left(\int_{\{w> \tau\}} (|w|^{p_{\beta,\mu}^*-p}|w|^{(k+1)p})^{\frac{p^*}{p_{\beta,\mu}^*}}\right)^{\frac{p_{\beta,\mu}^*}{p^*}}\right]\notag\\
		&\leq C\Bigg[\tau^{p_{\beta,\mu}^*-p} \|w\|^{(k+1)p}_{(k+1)\frac{pp^*}{p_{\beta,\mu}^*}}
		+\left(\int_{\{w> \tau\}} \left(|w|^{(p_{\beta,\mu}^*-p){\frac{p^*}{p_{\beta,\mu}^*}}}\right)^{\frac{p_{\beta,\mu}^*}{p_{\beta,\mu}^*-p}} dx\right)^{\frac{p_{\beta,\mu}^*-p}{p_{\beta,\mu}^*}\cdot\frac{p_{\beta,\mu}^*}{p^*}}\notag\\
		&\qquad\qquad \qquad\qquad\qquad\qquad\times\left(\int_{\{w> \tau\}} \left(|w|^{(k+1)p{\frac{p^*}{p_{\beta,\mu}^*}}}\right)^{\frac{p^*_{\beta,\mu}}{p}}dx\right)^{\frac{p}{p_{\beta,\mu}^*}\cdot\frac{p_{\beta,\mu}^*}{p^*}}\Bigg]\notag\\
		&\leq C\Bigg[\tau^{p_{\beta,\mu}^*-p} \|w\|^{(k+1)p}_{(k+1)\frac{pp^*}{p_{\beta,\mu}^*}}
		+\left(\int_{\{w> \tau\}} |w|^{p^*}\right)^{\frac{p_{\beta,\mu}^*-p}{p^*}}\left(\int_{\mb R^N} \|w|^{(k+1)p^*}dx\right)^{\frac{p}{p^*}}\Bigg]\notag\\
		&:=C_1 \|w\|^{(k+1)p}_{(k+1)\frac{pp^*}{p_{\beta,\mu}^*}}+ C_2(\tau)\|w\|^{(k+1)p}_{(k+1)p^*},
\end{align} where $\tau>0$ will be chosen later so that $C(\tau)>0$ will be sufficiently small.
Now plugging \eqref{r2},\eqref{r3}, \eqref{r4} \eqref{r5} in \eqref{r1} and letting $M\to\infty$ and applying Fatou's lemma, we get
\begin{align}\label{r7}
\|w\|^{(k+1)p}_{(k+1)p^*}&\leq 	\frac{(k+1)^p}{a(kp+1)S^{1/p}}\Bigg[\frac {b }{2^{\frac{(p-1)^2}{p}}}M^{kp+1}\|w\|^{2p}+  \la 2^{q/{2p}} M^{kp} \|f\|_{\frac{2p^*}{2p^*-q}}S^{-q/{2p}}\|w\|^{q/2}\notag\\
	&\qquad\qquad\qquad\qquad+C_1 \|w\|^{(k+1)p}_{(k+1)\frac{pp^*}{p_{\beta,\mu}^*}}+ C_2(\tau)\|w\|^{(k+1)p}_{(k+1)p^*}\Bigg]\notag\\
	&\leq 	\frac{(k+1)^p}{aS^{1/p}}\Big[ \tilde C_3(M+1)^{(k+1)p}+  C_1 \|w\|^{(k+1)p}_{(k+1)\frac{pp^*}{p_{\beta,\mu}^*}}+ C_2(\tau)\|w\|^{(k+1)p}_{(k+1)p^*}\Big],
\end{align} where the $\tilde C_3:=\tilde C_3(b,\la,p,N, q, \|f\|_{\frac{2p^*}{2p^*-q}},\|w\|)>0$ is a positive constant. Next, we can choose $\tau>0$ sufficiently large so that, by the Lebesgue dominated  convergence theorem we can find $C(\tau)<\frac {aS^{1/p}}{2(k+1)^p}$. Therefore, using this in \eqref{r7}, we obtain
\begin{align}\label{r8}
\|w\|_{(k+1)p^*}\leq \tilde C^{\frac{1}{(k+1)p}}	{(k+1)}^{\frac{1}{(k+1)}}\Big[ 1+ \|w\|_{(k+1)\frac{pp^*}{p_{\beta,\mu}^*}}\Big],
\end{align} where the $\tilde C:=\tilde C(a,b,M,\la,p,N, q, \|f\|_{\frac{2p^*}{2p^*-q}},\|w\|)>0$ is a positive constant. Since $p_{\beta,\mu}^*>p,$ we have $\frac{pp^*}{p_{\beta,\mu}^*}<p^*.$\\ 
	Case I: If there exists a sequence  ${k_n}$ such that $k_n\to\infty$ as $n\to\infty$ such that $$\|w\|_{(k_n+1)\frac{pp^*}{p_{\beta,\mu}^*}}\leq1,$$ then from \eqref{r8}, we can infer that
$\|w\|_{{\infty}}\leq1.$\\
Case II:   
If there is no such sequence satisfying the above condition as in Case I, then there exists $k_0>0$ such that
$$\|w\|_{(k+1)\frac{pp^*}{p_{\beta,\mu}^*}}>1,\text{\; for\;all\; } k\geq k_0.$$	Then  \eqref{r8} yields that 
\begin{align}\label{9}
	\|w\|_{{(k+1) p^*}}\leq \left(\tilde C^{\frac{1}{k+1}}\right)^{\frac{1}{p}} (k+1)^{\frac{1}{k+1}} \|w\|_{{\frac{(k+1) pp^*}{p_{\beta,\mu}^*}}}, \text{  for all } k\geq k_0.
\end{align}
Now we use standard bootstrap argument  by choosing the $1^{st}$ iteration as $k:=k_1$ in  \eqref{9} such that $(k_1+1)p=p_{\beta,\mu}^*$. In a similar manner, considering the $n^{th}$ iteration as $k=k_n:=k_{n-1}\frac{p_{\beta,\mu}^*}{p}$ to obtain 
\begin{align}\label{itn}
	\|w\|_{{(k_n+1) p^*}}&\leq \left(\tilde C^{\frac{1}{k_n+1}}\right)^{\frac{1}{p}} (k_n+1)^{\frac{1}{k_n+1}}\|w\|_{{ (k_{n-1}+1) p^*}}\notag\\
	&=\left(\tilde C^{\sum_{j=1}^{n}{\frac{1}{k_j+1}}}\right)^{\frac{1}{p}}\left(\prod_{j=1}^{n}(k_j+1)^{\sqrt{\frac{1}{k_j+1}}}\right)^{\sqrt{\frac{1}{k_j+1}}}\|w\|_{{ k_0 p^*}},
\end{align}
where $k_j+1=\left(\frac{p_{\mu,s}^*}{p}\right)^j.$ Since $\frac{p_{\mu,s}^*}{p}>1,$ we get $(k_j+1)^{\sqrt{\frac {1}
		{k_j+1}}}>1$ for all $j\in\mathbb N$ and
$\lim_{j\to\infty}(k_j+1)^{\sqrt{\frac {1}
		{k_j+1}}}=1.$
Hence,  there exists a constant $\underline C>1,$ independent of $n,$ such that $(k_j+1)^{\DD\sqrt{\frac{1}{k_j+1}}}<\underline C$ and thus, \eqref{itn} gives
\begin{align}\label{10}
	\|u\|_{{k_n p^*}}&\leq \left(\tilde C^{\DD\sum_{j=1}^{n}{\frac{1}{k_j+1}}}\right)^{\frac{1}{p}}{\underline C}^{\DD\sum_{j=1}^{n}\sqrt{\frac{1}{k_j+1}}}\|u\|_{{k_0  p^*}}.
\end{align}
As limit $n\to\infty,$  we have $$\sum_{j=1}^{\infty}{\frac{1}{k_j+1}}=\frac{p}{p_{\beta,\mu}^*-p};\;\;\;
\sum_{j=1}^{\infty}{\frac{1}{\sqrt{k_j+1}}}=\frac{\sqrt p}{\sqrt{ p_{\beta,\mu}^*}-\sqrt p}.$$
Thus, from  \eqref{10}, it follows that
\begin{align}\label{contr}
	\|w\|_{\al_n}\leq \left(\tilde C\right)^{\frac{1}{p_{\beta,\mu}^*-p}}\left(\underline C\right)^{\frac{\sqrt p}{\sqrt{p_{\beta, \mu}^*}-\sqrt {p}}}\|w\|_{{k_0  p^*}},
\end{align} where $\al_n:=(k_n+1)p^*$  and $\al_n\to\infty$ as $n\to\infty.$ Now we claim that  
\begin{align}\label{claim}
	w\in L^\infty(\mb R^N).
\end{align}
Indeed, if not then 
there exists $\e>0$ and a subset $\mathcal{D}$ of $\mb R^N$ with $\meas (\mathcal{D})>0$ such that 
$$w(x)>\left(\tilde C\right)^{\frac{1}{p_{\beta,\mu}^*-p}}\left(\underline C\right)^{\frac{\sqrt p}{\sqrt{p_{\beta, \mu}^*}-\sqrt {p}}}\|w\|_{k_0  p^*}+\vartheta\text{~~for } x\in\mathcal{D},$$
 which implies that
\begin{align*}
	\liminf_{\al_n\to\infty}\left(\int_{\mb R^N}|w(x)|^{\al_n}dx\right)^{\frac{1}{\al_n}}
	&\geq \DD\liminf_{\al_n\to\infty}\left(\int_{\mathcal{S}}|w(x)|^{\al_n}dx\right)^{\frac{1}{\al_n}}\\&\geq\DD\liminf_{\al_n\to\infty}\left(\left(\tilde C \right)^{\frac{1}{p_{\beta,\mu}^*-p}}\left(\underline C\right)^{\frac{\sqrt p}{\sqrt{p_{\beta, \mu}^*}-\sqrt {p}}}\|w\|_{{ k_0 p^*}}+\vartheta\right)\left(\meas(\mathcal{D})\right)^{\frac{1}{\al_n}}\\
	&=\left(\tilde C \right)^{\frac{1}{p_{\beta,\mu}^*-p}}\left(\underline C\right)^{\frac{\sqrt p}{\sqrt{p_{\beta, \mu}^*}-\sqrt {p}}}\|w\|_{k_0  p^*}+\vartheta.
\end{align*}
This contradicts \eqref{contr}. Thus, \eqref{claim} holds.\\
Now for the next part of the proposition, $f\in L^{\infty}(\mb R^N)$, hence using  Lemma \ref{L1}-$(g_3),(g_5),$ it follows that $$f(x)|h(w)|^{q-2}h(w) h'(w)\in L^{\infty}(\mb R^N).$$ Moreover, following the arguments in \cite{my1} (see also \cite{rs2}) in combination with Lemma \ref{L1}-$(g_5)$, we can deduce that $\int_{{\mathbb R^N}} \frac{|g(w(y))|^{2p_{\beta,\mu}^*}}{ |y|^\beta|x-y|^\mu}dy\in L^\infty(\mb R^N)$ and thus from \ref{L1}-$(g_3)$, it yields that 
\[\left(\int_{{\mathbb R^N}} \frac{|g(w(y))|^{2p_{\beta,\mu}^*}}{ |y|^\beta|x-y|^\mu}dy\right)\frac{|g(w(x))|^{2p_{\beta,\mu}^*}}{|x|^\beta} g'(w)\in L^{\infty}(\mb R^N).\] Therefore, using elliptic regularity theory, we infer that for any $R>0$ there exists $r(R)\in(0,1)$ such that $w\in C^{1,r}( {B}_R(0))$. This completes the proof of the lemma.

	
\end{proof}
	\noi {\bf Proof of Theorem \ref{main.result.1} :}
From the hypotheses, it follows that $ \mc I_\la$ is even and $\mc {I}_\la(0)=0.$ Also Lemma \ref{cc} ensures that $ \mc I_\la$ satisfies the  $(PS)_c$-condition for all $c<0$. But observe that, $\mc I_\la$ is not bounded from below in $D^{1,p}(\mb R^N)$. So, for applying Theorem \ref{sym mt}, we use a truncation technique.\\
 Let $w\in D^{1,p}(\mb R^N)$.  Using Lemma \ref{L1}-$(g_5),\,(g_6)$, \eqref{bc1} and \eqref{se}, we get\begin{align}\label{mpp}
\mc I_\la (w)&=\frac{a}{p}\displaystyle\int_{{\mathbb R^N}} |\nabla w|^{p}dx+\frac{ b}{2p}\left( \int_{\mathbb R^N}|g^{\prime}(w)|^p|\nabla w|^pdx\right)^2\nonumber\\ &\qquad\quad-\frac{\la}{q}\int_{\mathbb R^N}f(x)|g(w)|^q dx -\frac {1}{4p_{\beta,\mu}^*}\int_{{\mathbb R^N}}\left(\int_{{\mathbb R^N}} \frac{|g(w)|^{2p_{\beta,\mu}^*}}{ |y|^\beta|x-y|^\mu}dy\right)\frac{|g(w)|^{2p_{\beta,\mu}^*}}{|x|^\beta}dx\notag\\
	&\geq\frac{a}{p}\displaystyle\int_{{\mathbb R^N}} |\nabla w|^{p}dx-\frac{\la}{q} 2^{\frac {q}{2p}}\int_{\mathbb R^N}f(x)|w|^{q/2} dx - \frac {1}{4p_{\beta,\mu}^*} 2^{\frac{{p_{\beta,\mu}^*}}{p}}\int_{{\mathbb R^N}}\left(\int_{{\mathbb R^N}} \frac{|w|^{p_{\beta,\mu}^*}}{ |y|^\beta|x-y|^\mu}dy\right)\frac{|w|^{p_{\beta,\mu}^*}}{|x|^\beta}dx\notag\\
	& \geq \frac ap \|w\|^p-\frac \la q 2^{\frac {q}{2p}} \|f\|_{\frac{2p^*}{2p^*-p}}\|w\|^{q/2}_{p^*}-\frac{1}{4p_{\beta,\mu}^*} 2^{\frac{{p_{\beta,\mu}^*}}{p}}\|w\|_{\beta,\mu}^{2p_{\beta,\mu}^*}\notag\\&\geq \frac ap \|w\|^p-\frac \la q 2^{\frac {q}{2p}}S^{-\frac{q}{2p}} \|f\|_{\frac{2p^*}{2p^*-p}}\|w\|^{q/2}-\frac{1}{4p_{\beta,\mu}^*} 2^{\frac{{p_{\beta,\mu}^*}}{p}} S_{\beta,\mu}^{-\frac{2p_{\beta,\mu}^*}{p}}\|w\|^{2p_{\beta,\mu}^*}\notag\\
	&:= C_1 \|w\|^p-\la C_2\|w\|^{q/2}-C_3\|w\|^{{2p_{\beta,\mu}^*}}.
\end{align}  Define the function $\ell: \mb R_0^+\to\mb R$ as 
\begin{align}\label{t}
\mc \ell(t)=C_1 t^p-\la C_2t^{q/2}-C_3t^{{2p_{\beta,\mu}^*}}.
\end{align} 
Since $2<q<2p$, we can choose $\la_0$  sufficiently small such that for all $\la\in (0,\la_0)$ there exist $0<t_1<t_2$ so that $\ell<0$ in $(0,t_1)$, $\ell>0$ in $(t_1,t_2)$ and $\ell<0$ in $(t_2,\infty)$. Therefore $\ell(t_1)=0=\ell(t_2)$. 
Next, we choose a non-increasing function $\mc H\in C^\infty([0,\infty),[0,1])$ such that
\begin{equation*}
	\mc H(t)=
	\begin{cases}
		&1  \mbox{\;\; if}\ t\in [0,t_1] \\
		&0  \mbox{\;\; if}\ t\in [t_2,\infty).
	\end{cases}
\end{equation*}
and set $\Pi(w):=\mc H(\|w\|).$  Now we define the truncated functional $\hat I_\la: D^{1,p}(\mb R^N)\to \mb R$ of $\mc I_\la$ as
\begin{align}\label{ii}
	\hat  I_\la( u)&:= \frac{a}{p}\displaystyle\int_{{\mathbb R^N}} |\nabla w|^{p}dx+\frac{ b}{2p}\left( \int_{\mathbb R^N}|g^{\prime}(w)|^p|\nabla w|^pdx\right)^2\\ &\qquad\quad-\Pi (w)\frac{\la}{q}\int_{\mathbb R^N}f(x)|g(w)|^q dx -\Pi(w)\frac {1}{4p_{\beta,\mu}^*}\int_{{\mathbb R^N}}\left(\int_{{\mathbb R^N}} \frac{|g(w)|^{2p_{\beta,\mu}^*}}{ |y|^\beta|x-y|^\mu}dy\right)\frac{|g(w)|^{2p_{\beta,\mu}^*}}{|x|^\beta}dx.
\end{align} Then, it can be verified easily that $\hat I$ satisfies the following:
\begin{enumerate}
\item  $\hat I_\la\in C^1(D^{1,p}(\mb R^N),\mb R) $, $\hat  I_\la( 0)=0$. 
\item  $\hat I_\la$ is even, coercive and bounded from below in $D^{1,p}(\mb R^N)$.
\item Let $c<0$, then there exists $\la_1>0$ such that for all $\la\in (0,\la_1)$, $\hat I_\la$ satisfies the Palais-Smale condition.
\item If $\hat  I_\la( w)<0,$ then $\|w\|\leq t_1$ and $\hat I_\la(w)=\mc I_\la(w)$.
\end{enumerate}
 For any $k\in\mb N$, we consider $k$ numbers of disjoint open sets denoted by $V_j$, $j=1,2,\cdots k$ with $\cup_{j=1}^k V_j\subset \Om$, where $\Om\not=\emptyset$ is given as in Theorem \ref{main.result.1}. Now we choose $w_j\in D^{1,p}(\mb R^N)\cap C_0^\infty(V_j)\setminus\{0\}$, with $\|w_j\|=1$ for each $j=1,2,\cdots,k$. Set $$X_k=span\{w_1,w_2,\cdots,w_k\}.$$  Now we claim that there exists $0<\varrho_k<t_1,$ sufficiently small such that \begin{align}\label{claim1}
 	m_k:=\max\{\hat I_\la(u):u\in X_k,~\|w\|=\varrho_k\}\leq0.
 \end{align}
  Suppose that \eqref{claim1} does not hold. Then there exists a sequence $\{w_n\}:=\{w_n^{(k)}\}$ in $X_k$ such that 
\begin{align}\label{f11}\|w_n\|\to\infty;\; \hat I_\la(w_n)\geq 0.
\end{align} Let's set\[u_n=\frac{w_n}{\|w_n\|}.\] Then $u_n\in D^{1,p}(\mb R^N)$ and $\|u_n\|=1.$ Since $X_k$ is finite dimensional, there exists $u\in X_k\setminus\{0\}$ such that
\begin{align*}
	u_n&\to u \text{\;\; strongly with respect to \;\;} \|\cdot\|;\\
	u_n(x)&\to u(x) \text{\;\; a.e. in \;\;} \mb R^N.
\end{align*} As $u\not\equiv 0$, we get $|w_n(x)|\to\infty$ as $n\to\infty$. Thus, as $n\to\infty$,
\[\frac{1}{\|w_n\|^{2p}}\int_{\mb R^N}\int_{\mb R^N} \frac{|w_n(x)|^{{p_{\beta,\mu}^*}}|w_n(y)|^{{p_{\beta,\mu}^*}}}{|x|^\beta |x-y|^\mu |y|^\beta}dxdy=\int_{\mb R^N}\int_{\mb R^N} \frac{|w_n(x)|^{{p_{\beta,\mu}^*}-p}|w_n(y)|^{{p_{\beta,\mu}^*}-p}}{|x|^\beta |x-y|^\mu |y|^\beta} |u_n(x)|^{p}|u_n(y)|^{p}dxdy\to \infty.\] Using this together with Lemma \ref{L1}-$(g_3),(g_8)$, from \eqref{ii}, we obtain
\begin{align*}
	\hat I_\la(w_n)&\leq \frac ap \|w_n\|^p+\frac{b}{2p}\|w_n\|^{2p}-\frac{(g(1))^{4p_{\beta,\mu}^*}}{4{p_{\beta,\mu}^*}}\int_{\mb R^N}\int_{\mb R^N} \frac{|w_n(x)|^{{p_{\beta,\mu}^*}}|w_n(y)|^{{p_{\beta,\mu}^*}}}{|x|^\beta |x-y|^\mu |y|^\beta}dxdy\\
	&\leq \|w_n\|^{2p}\left((\frac ap +\frac{b}{2p})-\frac{(g(1))^{4p_{\beta,\mu}^*}}{4{p_{\beta,\mu}^*}}\frac{1}{\|w_n\|^{2p}}\int_{\mb R^N}\int_{\mb R^N} \frac{|w_n(x)|^{{p_{\beta,\mu}^*}}|w_n(y)|^{{p_{\beta,\mu}^*}}}{|x|^\beta |x-y|^\mu |y|^\beta}dxdy\right)\\
	&\to-\infty
\end{align*} as $n\to\infty.$
This contradicts \eqref{f11}. Thus, the claim is proved. Now choose $A_k:=\{w\in X_k\; :\; \|w\|=\varrho_k\}$. Clearly $\gamma (A_k)=k$ and $A_k$ is closed and symmetric, and hence $A_k\in\Sigma_k$ and also from \eqref{claim1}, $\sup_{w\in A_k}\hat I_\la(w)<0.$ Therefore, $\hat I_\la $ satisfies all the assumption in Theorem \ref{sym mt}. Thus, $\hat I_\la$ admits a sequence of critical points $\{w_k\}$ in $D^{1,p}(\mb R^N)$ such that $w_k\not =0$, $\hat I_\la (w_k)\leq 0$ for each $k\in \mb N$  and $\|w_k\| \to 0$ as $k\to\infty$. So, for $t_1>0$, there exists $ k_0\in\mathbb N$ such that for all $k\geq k_0$
it follows that $\|w\|<t_1$ which yields that $\hat I_\la(w_k)= \mc I_\la(w_k)$ for all $k> k_0.$ This together with Proposition \ref{infty} concludes the proof of the theorem.
	\section{\bf Proof of Theorem \ref{sym-infinite-sol}}
\noi Before proceeding into the proof of Theorem \ref{sym-infinite-sol}, first we recall the following $\mathbb Z_2$-symmetric version of mountain pass theorem due to \cite {sm}.
	\begin{theorem}\label{sm}
	Let $X$ be an infinite dimensional Banach space with $X= Y\oplus Z$, where $Y$ is finite dimensional and let \(\mc J \in C^{1}(X,\mathbb R)\) be an even functional with $\mc J(0)=0$ such that the following conditions hold:
	\begin{enumerate}
		\item[$(\mc B_1)$] there exist positive constants \(l>0, \mc K>0\) such that $\mc J(u)\geq\mc K$ for all \(u\in \partial B_{l}(0) \cap Z\);
		\item[$(\mc B_2)$] there exists $c^*> 0$ such that \(\mc J\) satisfies the $(PS)_c$ condition for $0<c< c^*$;
		\item[$(\mc B_3)$] for any finite dimensional subspace $\hat{X}\subset X$, there is \(R= R(\hat{X})>0\) such that $\mc J(u)\leq 0$ for all \(u\in\tilde{X}\setminus B_{R}(0)\).
	\end{enumerate}
	Assume that $Y$ is $k$-dimensional and $Y= span\{v_1, v_2, \cdots, v_k\}$. For $n \geq k$, inductively choose \(v_{n+1}\not\in Y_n := span\{v_1, v_2, \cdots, v_n\}.\) Let \(R_n =R(Y_n)\) and \(D_n = B_{R_n(0)}\cap Y_n\). Define
	\[G_n=\{h \in C(D_n, X): h|_{\partial B_{R_n}(0)},\; h \;\mbox{is odd and}\; h(u)=u, \text { for all } B_{R_n}(0)\cap Y_n \}\]
	and
	\begin{align}\label{gm}\Gamma_j = \{h(\overline{D_n\setminus S}) : h \in G_n, n\geq j,\; S \text{\; is closed and symmetric},\;\mbox{and}\; \gamma(S)\leq n-j\},\end{align}
	where $\ga(S)$ is Krasnoselskii's genus of $S$. For each $j\in \mathbb N$, set
	\[c_j :=\inf_{A\in \Gamma_j} \max_{u\in A} \mc J(u).\]
	Thus \(0<\alpha\leq c_j\leq c_{j+1}\) for $j>k$ and  if $j>k$ and $c_j< c^*$, then we conclude that \(c_j\) is the critical value of $\mc J$. Furthermore, if $c_j=c_{j+1}=\cdots=c_{j+m}=c< c^*$ for $j>k$, then $\gamma(K_c)\geq m+1$, where
	\[K_c= \{u\in X : \mc J(u)=c\; \mbox{and} \;\mc J^{\prime}(u)=0\}.\]
\end{theorem}		
\noi Now we show that $\mc I_\la $ satisfies all the hypotheses of Theorem \ref{sm}, when $q=2p$. 
\begin{lemma}\label{smg}
Let  $q=2p$ and \eqref{assumption1} hold.	Then  $\mc I_\la$ satisfies the conditions $(\mc B_1)$-$(\mc B_3)$ of Theorem \ref{sm} for all $\la\in (0,\,aS\|f\|^{-1}_{\frac{p^*}{p^*-p}}). $
\end{lemma}
\begin{proof}
{\it Verification of $(\mathcal B_1):$} For $w\in D^{1,p}(\mb R^N),$ arguing similarly as in \eqref{mpp}, we have
\begin{align*}
	\mc I_\la( w)&\geq  \frac{\|w\|^p}{p}\left( a- \la S^{-1} \|f\|_{\frac{p^*}{p^*-p}}\right)-\frac{1}{4p_{\beta,\mu}^*} 2^{\frac{{p_{\beta,\mu}^*}}{p}} S_{\beta,\mu}^{-\frac{2p_{\beta,\mu}^*}{p}}\|w\|^{2p_{\beta,\mu}^*}\\
\end{align*} Now for $\la<aS\|f\|^{-1}_{\frac{p^*}{p^*-p}}$,  we can choose $\|w\|=l<<1$ such that $\mc I_\la( w)\geq \mc K>0.$\\\\
{\it Verification of $(\mathcal B_2):$} It follows from Lemma \ref{ce}.\\\\
{\it Verification of $(\mathcal B_3):$} To show this, first claim that for any finite dimensional subspace
$Y$ of $D^{1,p}(\mb R^N)$ there exists $ R_0= R_0(Y)$ such that $ \mc I_\la( w)<0$ for all $w\in D^{1,p}(\mb R^N)\setminus B_{ R_0} (Y),$ where $B_{R_0}(Y)=\{w \in D^{1,p}(\mb R^N): \|w\|\leq  R_0\}.$ Fix $\phi\in D^{1,p}(\mb R^N),\;\|\phi\|=1.$ For $t>1$, using Lemma \ref{L1}-$(g_3), (g_8)$, 
we get
\begin{align}\label{sm1}
	\mc I_\la( t\phi)&\leq \frac ap t^{p} \|\phi\|^p+  t^{2p}\frac{b}{2p}\|\phi\|^{2p}
	-\frac{1}{{4p_{\beta,\mu}^*}} (g(1))^{{4p_{\beta,\mu}^*}}t^{2p_{\beta,\mu}^*}\int_{\mb R^N}\int_{\mb R^N} \frac{|\phi(x)|^{{p_{\beta,\mu}^*}}|\phi(y)|^{{p_{\beta,\mu}^*}}}{|x|^\beta |x-y|^\mu |y|^\beta}dxdy\notag\\
	&\leq C_4  t^{2p}\|\phi\|^{2p}
	-C_5 t^{2p_{\beta,\mu}^*}\|\phi\|_{\beta,\mu}^{2p_{\beta,\mu}^*}
\end{align}
Since $Y$ is finite dimensional all norms are equivalent on $Y$, which yields that there exists some constant $C(Y)>0$ such that $C(Y)\|\phi\|\leq\|\phi\|_{\beta,\mu}.$ Therefore from \eqref{sm1}, we obtain

\begin{align*}
	\mc I_\la( tw)&\leq C_4  t^{2p}
	-C_5 (C(Y))^{2p_{\beta,\mu}^*}t^{2p_{\beta,\mu}^*}\|\phi\|^{2p_{\beta,\mu}^*}\\
	&= C_4  t^{2p}
	-C_5 (C(Y))^{2p_{\beta,\mu}^*} t^{2p_{\beta,\mu}^*}\to-\infty
\end{align*}
as $t\to\infty$.
Hence, there exists $ R_0>0$ large enough such that $\mc I_\la( w)<0$ for all
$w\in D^{1,p}(\mb R^N)$ with $\|w\|={R}$ and $ R\geq R_0$. Therefore $\mc I_\la$ satisfies the assertion $(\mathcal B_2)$.
\end{proof}
\begin{lemma}\label{ss}
There exists a non-decreasing sequence $\{s_n\}$ of positive real numbers, independent of $\la$ such that for any $\la>0$, we have
\[c_n^\la:=\inf_{A\in \Gamma_n}\max_{w\in A} \mc I_\la(w)<s_n,\] where $\Gamma_n$ is defined in \eqref{gm}.
\end{lemma}
\begin{proof}
	Recalling the definition of $c_n^\la$ and using  Lemma \ref{L1}-$(g_3),(g_8)$, from \eqref{ii}, we get \[ c_n^\la\leq \inf_{A\in \Gamma_n}\max_{w\in A}\left[\frac ap \|w_n\|^p+\frac{b}{2p}\|w_n\|^{2p}-\frac{(g(1))^{4p_{\beta,\mu}^*}}{4{p_{\beta,\mu}^*}}\|w\|_{\beta,\mu}^{2p_{\beta,\mu}^*}\right]:=s_n\] Then clearly from the definition of $\Gamma_n$, it follows that $s_n<\infty$ and $s_n\leq s_{n+1}$.
\end{proof}
\noi {\bf Proof of Theorem \ref{sym-infinite-sol}:} From the hypotheses of the theorem it follows that $\mc I_\la$	is even and we have $\mc I_\la(0)=0.$  Now we argue similarly as in \cite{sm}.
From the Lemma \ref{ss}, we can choose, $\hat a>0$ sufficiently large such that for any $a>\hat a$,
\[\sup_n s_n<\frac {1}{4p}\left( aS_{\ba,\mu}\right)^{\frac{p_{\ba,\mu}^*}{p_{\ba,\mu}^*-1}}:= c^*,\] that is,
\[c_n^\la<s_n<\frac {1}{4p}\left( aS_{\ba,\mu}\right)^{\frac{p_{\ba,\mu}^*}{p_{\ba,\mu}^*-1}}.\] Hence, for all $\la\in(0,\,aS\|f\|^{-1}_{\frac{p^*}{p^*-p}})$ and $a>\hat a$, we have
\[0<c_1^\la\leq c_2^\la\leq\cdots\leq c_n^\la<s_n<c^*.\] Now by Theorem \ref{sm}, we infer that the levels $c_1^\la\leq c_2^\la\leq\cdots\leq c_n^\la$ are critical values of $\mc I_\la.$ Therefore, if $c_1^\la<c_2^\la<\cdots< c_n^\la$, then $\mc I_\la$ has at least $n$ number of critical points. Furthermore, if $c_j^\la=c_{j+1}^\la$ for some $j=1,2,\cdots,k-1$, then again Theorem \ref{sm} yields that $A_{c_j^\la}$ is an infinite set. Hence, \eqref{2p} has infinitely many solutions. 
Therefore, we can conclude that \eqref{2p} has at least $n$ pair of solutions Since $n$ is arbitrary, we get infinitely many solutions  and moreover, these solutions are in $L^\infty(\mb R^N)$ by Proposition \ref{infty}.
\section{\bf Proof of Theorem \ref{dual-fount-sol}}
\noi In this section, we prove Theorem \ref{dual-fount-sol} using Theorem \ref{sm}. For that, first we show $\mc I_\la$ verifies all the hypotheses of Theorem \ref{sm}, when $2p<q<2p^*$.
\begin{lemma}\label{ft}
Let $2p<q<2p^*$ and \eqref{assumption1} hold.	Then  $\mc I_\la$ satisfies the conditions $(\mc B_1)$-$(\mc B_3)$ of Theorem \ref{sm} for all $\la\geq 0.$
\end{lemma}
\begin{proof}
{\it Verification of $(\mathcal B_1):$}  Let $w\in D^{1,p}(\mb R^N)$ with $\|w\|<1$.  Using the similar arguments as in \eqref{mpp},  we get\begin{align*}
					\mc I_\la( w)
					&\geq \frac ap \|w\|^p-\frac \la q 2^{\frac {q}{2p}}S^{-\frac{q}{2p}} \|f\|_{\frac{2p^*}{2p^*-p}}\|w\|^{q/2}-\frac{1}{4p_{\beta,\mu}^*} 2^{\frac{{p_{\beta,\mu}^*}}{p}} S_{\beta,\mu}^{-\frac{2p_{\beta,\mu}^*}{p}}\|w\|^{2p_{\beta,\mu}^*}.
				\end{align*} 
		 Since $ 2p<q$ and $p<p_{\beta,\mu}^*$, we can choose $0<\rho<1$  sufficiently small so that, we  obtain  for all $w\in D^{1,p}(\mb R^N)$ with $\|w\|=\rho$, $\mc I_\la( w) \geq \al >0$ for some $\al>0$ depending on $\rho$.\\\\
{\it Verification of $(\mathcal B_2):$} It follows from Lemma \ref{cf}, since $c^{**}>0$.\\\\
{\it Verification of $(\mathcal B_3):$} The argument follows similarly as in {\it Verification of $(\mathcal B_3)$} in Lemma \ref{smg}.
\end{proof}
\noi {\bf Proof of Theorem \ref{dual-fount-sol}} Using Lemma \ref{ft} and arguing in a similar fashion as in Lemma \ref{ss} and as in \ref{sm}, we can conclude that \eqref{2p} has at least $n$ pairs of distinct solutions for all $\la>0$. Since $n$ is arbitrary, we have infinitely many solutions.  Now Proposition \ref{infty} yields that these solutions belong to $L^\infty(\mb R^N)$.


\begin{thebibliography}{99}
\bibitem{my3}C. O. Alves,  F. Gao, M. Squassina, and M. Yang, Singularly perturbed critical Choquard equations, J.  Differential Equations, 263 (2017), 3943-3988.
				\bibitem{rs1} R. Biswas and S. Tiwari, { Regularity results for Choquard equations involving fractional $p$-Laplacian}, Accepted in Math. Nachr., 2022. https://arxiv.org/abs/2008.07398.
				\bibitem{rs2} R. Biswas, S. Goyal and K. Sreenadh, { Quasilinear Schr\"odinger equations with Stein-Weiss type convolution and critical exponential nonlinearity in $\mathbb R^N$}, 2022.
				https://arxiv.org/abs/2202.07611.
				
					\bibitem{bass}	F. Bass and N. N. Nasanov, {\it Nonlinear electromagnetic-spin waves}, Phys. Rep. 189 (1990), 165-223.
				\bibitem{bz} H. Brezis, {\it Functional analysis, Sobolev spaces and partial differential equations}, Universitext, Springer, New York, 2011.
				
				
				\bibitem{my6}S. Chen, C. A. Santos, M. Yang, and J. Zhou, Global multiplicity of solutions for a quasilinear elliptic equation with concave and convex nonlinearities, Adv. Differ. Equ., 26 (2021), 425-458.
				\bibitem{bz1}B.  Cheng, J. Chen  and B. Zhang, {\it Least energy nodal solutions for Kirchhoff-type Laplacian problems}, Math. Methods Appl. Sci., 43 (2020), 3827--3849.
				\bibitem{CJ}M. Colin, L. Jeanjean, {\it Solutions for a quasilinear Schr\"odinger equation: a dual approach}, Nonlinear Anal. 56 (2004), 213–226.
				\bibitem{bose}
				{F.~Dalfovo, S.~Giorgini, L.~P. Pitaevskii and S.~Stringari},
				{\it Theory of Bose-Einstein condensation in trapped gases},
				{Rev. Modern Phys.},
				{ 71} ({1999}),
				{463}.
				
				
				\bibitem{do} J. M. do Ó, O. H. Miyagaki, S. H. M. Soares, {\it Soliton solutions for quasilinear Schrödinger equations: The critical exponential case}, Nonlinear Anal., 67
				(2007), 3357-3372.
				\bibitem{8} J. M. do Ó, O. H. Miyagaki and S. H. M. Soares, {\it Soliton solutions for quasilinear Schr\"odinger
					equations with critical growth}, J. Differential Equations, 248 (2010), 722-744.
				\bibitem{my1}L. Du, F. Gao and M. Yang, {\it On elliptic equations with Stein–Weiss type convolution parts}, Math. Z., 301 (2022), 2185--2225.
				 
			\bibitem{my4}F. Gao, M. Yang, and J. Zhou, Existence of multiple semiclassical solutions for a critical Choquard equation with indefinite potential, Nonlinear Anal., 195 (2020), 111817.
			
				\bibitem{has}R. W. Hasse, {\it A general method for the solution of nonlinear soliton and kink Schr\"odinger equation}, Z. Phys., B 37 (1980), 83-87.
				\bibitem{12} L. Jeanjean, T. J. Luo and Z. Q. Wang, {\it Multiple normalized solutions for quasi-linear
					Schr\"odinger equations}, J. Differential Equations, 259 (2015), 3894-3928.
				
					\bibitem{fig2} G.M. Figueiredo and J.R. Santos J\'{u}nior, {\it Multiplicity of solutions for a Kirchhoff equation with subcritical or critical growth}, Differ. Integral Equ., 25 (2012), 853-868.
				
				
				\bibitem{kajikiya} R. Kajikiya, {\it A critical point theorem related to the symmetric mountain pass lemma and its applications to elliptic equations}, J. Funct. Anal., 225 (2005), 352-370.
				\bibitem{1}S. Kurihara, {\it Large-amplitude quasi-solitons in superfluids films}, J. Phys. Soc. Japan, 50 (1981) 3262-3267.
				
					\bibitem{lg1}S. Liang, P. Pucci and B. Zhang {\it Multiple solutions for critical  Choquard-Kirchhoff type
				equations}, Adv. Nonlinear Anal., 10 (2021), 400-419.
			(1976/77), 93-105.
			\bibitem{bz3} S. Liang, L. Wen and B. Zhang, {\it Solutions for a class of quasilinear Choquard equations with Hardy–Littlewood–Sobolev critical nonlinearity}, Nonlinear Anal., 198 (2020), 111888.
			\bibitem{lg2}S. Liang and J. Zhang {\it Multiple solutions for Kirchhoff type problems with critical growth in $\mb R^N$
				equations}, Rocky Mountain J. Math. , 47 (2017), 527-551.
		\bibitem{lgg} S. Liang and  Y. Song, {\it Nontrivial solutions of quasilinear Choquard equation involving the $ p $-Laplacian operator and critical nonlinearities}, Differ.  Integral Equ., 35 (2022), 359-370.
		\bibitem{bz2} C. Lei, Y. Lei and B. Zhang, {\it Solutions for critical Kirchhoff-type problems with near resonance}, J. Math. Anal. Appl., 513 (2022), 126205.
			\bibitem{choq} E. H. Lieb, {\it Existence and uniqueness of the minimizing solution of Choquard
				nonlinear equation}, Studies  Appl. Math., 57 (1976/77), 93-105.
				
				\bibitem{lieb} E. Lieb and M. Loss, {\it Analysis},  Graduate Studies in Mathematics, AMS, Providence, Rhode island, 2001.
				\bibitem{Lions} P. L. Lions, {\it The concentration compactness principle in the calculus of variations part-I}, Rev. Mat. Iberoamericana, 1 (1985), 185-201.
				\bibitem{pl1} P.L. Lions. {\it The concentration-compactness principle in the calculus of variations, The locally compact case, Part 1}, Ann. Inst. H. Poincar\'e Anal. Non Lin\'eaire, 1 (1984), 109–145.
				\bibitem{24} X. Q. Liu, J. Q. Liu and Z. Q. Wang, {\it Quasilinear elliptic equations with critical growth via
					perturbation method}, J. Differential Equations, 254 (2013), 102-124.
				
				\bibitem{lu}
				{	D.~L\"{u}},
				{\it A note on Kirchhoff-type equations with Hartree-type
					nonlinearities},
				{Nonlinear Anal.}, { 99} (2014), 35--48.
				
				\bibitem{moroz-main} V. Moroz and J. Van Schaftingen, { \it Existence of groundstates for a class of nonlinear Choquard equations},  Trans. Amer. Math. Soc., {367} (2015), 6557--6579.
				
%
			
				\bibitem{pekar} S. Pekar, {\it Untersuchung \"uber die Elektronentheorie der Kristalle}, Akademie
				Verlag, Berlin (1954).
				
				
				
				
				%
				
				%
				%
				
				%
				%
				%
				%
				%
				%
				%
				%
				%
				%
				%
				\bibitem{penrose}
				{R.~Penrose}, {\it Quantum computation,
					entanglement and state reduction}, {R. Soc. Lond. Philos.
					Trans. Ser. A Math. Phys. Eng. Sci.},
				{356} ({1998}), {1927--1939}.
				
				\bibitem{sm} P. H. Rabinowitz, Minimax methods in critical point theory with applications to differential equations, CBMS regional conference series in Mathematics, vol. 65, American Mathematical Society, Providence,  1986.
				\bibitem{33}B. Ritchie, {\it Relativistic self-focusing and channel formation in laser-plasma interactions}, Phys. Rev. E, 50 (1994), 687-689.
				\bibitem{my5} C. A. Santos, M. Yang, and J. Zhou, Global multiplicity of solutions for a modified elliptic problem with singular terms, Nonlinearity, 34 (2021), 7842.
				\bibitem{29} D. Ruiz and G. Siciliano, {\it Existence of ground states for a modified nonlinear Schr\"odinger
					equation}, Nonlinearity, 23 (2010), 1221-1233.
			
				
				
			
				
			
				
				
			
				
				
			
			
				
				
%
			
				
				
			
				
				
			
				
				
			\bibitem{na} Y.	Song,  F. Zhao, H. Pu and S. Shi, {\it Existence results for Kirchhoff equations with Hardy–Littlewood–Sobolev critical nonlinearity}, Nonlinear Anal., 198 (2020), 111900.
				\bibitem{SW} E. M. Stein and G. Weiss, {\it Fractional integrals on $n$-dimensional Euclidean space}, J. Math. Mech.,  7 (1958), 503--514.
			\bibitem{bz4} M.	Sun, J. Su and B. Zhang, {\it Critical groups and multiple solutions for Kirchhoff type equations with critical exponents}, Commun. in Contemp. Math., 23 (2021), 2050031.
				\bibitem{my2} M. Yang and X. Zhou, On a Coupled Schrödinger System with Stein–Weiss Type Convolution Part, J. Geom. Anal., 31 (2021), 10263–10303.
			\end{thebibliography}
		\end{document}